\documentclass[reqno]{amsart}

\usepackage{amssymb,amsmath,amsthm,amsfonts,mathtools,bbm,mathrsfs,enumitem,graphics,float,multirow}
\usepackage[normalem]{ulem}
\usepackage{hyperref}

\hypersetup{
    colorlinks,
    linkcolor={blue!80!black},
    citecolor={red!80!black},
    urlcolor={blue!80!black}
}
\usepackage{xargs}                      
\usepackage{xcolor}  
\usepackage[colorinlistoftodos,prependcaption,textsize=tiny]{todonotes}
\newcommandx{\unsure}[2][1=]{\todo[linecolor=red,backgroundcolor=red!25,bordercolor=red,#1]{#2}}
\newcommandx{\change}[2][1=]{\todo[linecolor=blue,backgroundcolor=orange!25,bordercolor=blue,#1]{#2}}
\newcommandx{\info}[2][1=]{\todo[linecolor=OliveGreen,backgroundcolor=OliveGreen!25,bordercolor=OliveGreen,#1]{#2}}

\usepackage[top=1.6in,bottom=1.5in,left=1.4in,right=1.4in]{geometry}
\allowdisplaybreaks
\usepackage{cellspace}
\newcommand\redout{\bgroup\markoverwith{\textcolor{red}{\rule[0.5ex]{2pt}{0.8pt}}}\ULon}

\newtheorem{theorem}{Theorem}
\newtheorem{lemma}[theorem]{Lemma}
\newtheorem{claim}[theorem]{Claim}
\newtheorem{fact}[theorem]{Fact}

\usepackage{caption}
\usepackage{comment}
\usepackage{tikz}
\usetikzlibrary{calc}

\tikzset{font= }

\newcommand\nc\newcommand
\nc\bfa{{\boldsymbol a}}\nc\bfA{{\boldsymbol A}}\nc\cA{{\mathscr A}}
\nc\bfb{{\boldsymbol b}}\nc\bfB{{\boldsymbol B}}\nc\cB{{\mathscr B}}
\nc\bfc{{\boldsymbol c}}\nc\bfC{{\boldsymbol C}}\nc\cC{{\mathscr C}}
\nc\bfd{{\boldsymbol d}}\nc\bfD{{\boldsymbol D}}\nc\cD{{\mathscr D}}
\nc\bfe{{\boldsymbol e}}\nc\bfE{{\boldsymbol E}}\nc\cE{{\mathscr E}}
\nc\bff{{\boldsymbol f}}\nc\bfF{{\boldsymbol F}}\nc\cF{{\mathscr F}}
\nc\bfg{{\boldsymbol g}}\nc\bfG{{\boldsymbol G}}\nc\cG{{\mathscr G}}
\nc\bfh{{\boldsymbol h}}\nc\bfH{{\boldsymbol H}}\nc\cH{{\mathscr H}}\nc\fH{{\mathfrak H}}
\nc\bfi{{\boldsymbol i}}\nc\bfI{{\boldsymbol I}}\nc\cI{{\mathcal I}}
\nc\bfj{{\boldsymbol j}}\nc\bfJ{{\boldsymbol J}}\nc\cJ{{\mathscr J}}
\nc\bfk{{\boldsymbol k}}\nc\bfK{{\boldsymbol K}}\nc\cK{{\mathscr K}}
\nc\bfl{{\boldsymbol l}}\nc\bfL{{\boldsymbol L}}\nc\cL{{\mathscr L}}
\nc\bfm{{\boldsymbol m}}\nc\bfM{{\boldsymbol M}}\nc\cM{{\mathscr M}}
\nc\bfn{{\boldsymbol n}}\nc\bfN{{\boldsymbol N}}\nc\sN{{\mathscr N}}
\nc\bfo{{\boldsymbol o}}\nc\bfO{{\boldsymbol O}}\nc\cO{{\mathscr O}}
\nc\bfp{{\boldsymbol p}}\nc\bfP{{\boldsymbol P}}\nc\cP{{\mathscr P}}
\nc\bfq{{\boldsymbol q}}\nc\bfQ{{\boldsymbol Q}}\nc\cQ{{\mathscr Q}}
\nc\bfr{{\boldsymbol r}}\nc\bfR{{\boldsymbol R}}\nc\cR{{\mathscr R}}
\nc\bfs{{\boldsymbol s}}\nc\bfS{{\boldsymbol S}}\nc\cS{{\mathscr S}}
\nc\bft{{\boldsymbol t}}\nc\bfT{{\boldsymbol T}}\nc\cT{{\mathscr T}}
\nc\bfu{{\boldsymbol u}}\nc\bfU{{\boldsymbol U}}\nc\cU{{\mathscr U}}
\nc\bfv{{\boldsymbol v}}\nc\bfV{{\boldsymbol V}}\nc\cV{{\mathscr V}}
\nc\bfw{{\boldsymbol w}}\nc\bfW{{\boldsymbol W}}\nc\cW{{\mathscr W}}
\nc\bfx{{\boldsymbol x}}\nc\bfX{{\boldsymbol X}}\nc\cX{{\mathscr X}}
\nc\bfy{{\boldsymbol y}}\nc\bfY{{\boldsymbol Y}}\nc\cY{{\mathscr Y}}
\nc\bfz{{\boldsymbol z}}\nc\bfZ{{\boldsymbol Z}}\nc\cZ{{\mathscr Z}}

\renewcommand{\le}{\leqslant}
\renewcommand{\leq}{\leqslant}

\renewcommand{\geq}{\geqslant}

\nc{\Cay}{{\sf Cay}}

\usepackage{pgfplots}
\usepackage{tkz-euclide}

\nc{\ff}{{\mathbb F}}

\newcommand\remove[1]{}

\title{The Time of Bootstrap Percolation in high dimensions}
\author[]{Fengxing Zhu}\thanks{Institute for Systems Research and Department of ECE, University of Maryland, College Park, MD 20742, USA, fengxing@terpmail.umd.edu. Supported in part by NSF grant CCF 2330909.}

\date{}

\begin{document}
\begin{abstract}
 We consider the $d$-neighbor bootstrap percolation process on the $d$-dimensional torus, with vertex set $V=\{1,\cdots,n\}^d$ and edge set $\{xy:\sum_{i=1}^d|x_i-y_i (\text{mod} \; n)|=1\}$.
 We determine the percolation time up to a constant factor with high probability when the initial infection probability is in a certain range and the infection threshold is $d$, extending one of the two main theorems from Balister,Bollob{\'a}s, and Smith (2016) about the percolation time with the infection threshold equal to $2$ on the two-dimensional torus.  
\end{abstract}
\maketitle

\section{Introduction}
The process of $r$-neighbor bootstrap percolation on an undirected graph $G(V,E)$, with an integer $r\geq 1$, was introduced by Chalupa, Leith, and Reich \cite{Chalupa_1979}. In this process, each vertex is either infected or healthy, and once a vertex becomes infected, it remains infected forever. Initially, a set of vertices $A_0$ is infected, and let $A_i$ denote the set of infected vertices up to step $i$. The bootstrap percolation process evolves in discrete steps as follows: for $i>0$,
$$A_i = A_{i-1} \cup \{v \in V:|N(v) \cap A_{i-1}| \geq r \},$$
where $N(v)$ represents the neighborhood of the vertex $v$. In simple terms, a vertex that is not initially infected becomes infected at step $i$ if it has at least $r$ infected neighbors at step $i-1$. Here, we define $r$ as the infection threshold for all $v \in V$ and a contagious set as a set of initially infected vertices that leads to the complete infection of the entire graph. Percolation occurs if all vertices have been infected by the end of the process. In random bootstrap percolation, each vertex is independently and randomly infected at the beginning with a probability $p$. The central question in the problem of random bootstrap percolation is to determine the critical probability, denoted as $p_c(G,r)$,
$$p_c(G,r)=\sup\{p \in (0,1): \mathbb{P}_p(A \; \text{percolates on} \: G) \leq \frac{1}{2} \},$$
where $A$ represents the set of initially infected vertices. 

 Much research has been dedicated to examining the critical probability on $d$-dimensional grid graphs $[n]^d$. Notable works include \cite{Aizenman_1988,Cirillo,CERF200269,Holroyd2002SharpMT},
with the work in \cite{Balogh} establishing a sharp estimate for $p_c([n]^d,r)$, $2 \leq r \leq d$.

In addition to grid graphs the critical probability on the $n$-dimensional binary hypercube $Q_n$ has been investigated for various infection thresholds. In \cite{balogh_bollobas_2006}, the authors obtained a tight estimate for the critical probability $p_c(Q_n,2)$ up to a constant factor, with a sharper estimate later provided in \cite{balogh_bollobas_morris_2010}. In addition to studying the critical probability when the infection threshold is a constant, \cite{balogh_bollobas_morris_2009} derived a sharp estimate for $p_c(Q_n,\frac{n}{2})$ and the exact second-order term. More recently, the main result from \cite{balogh_bollobas_2006} on the binary Hamming cube was extended to the $q$-ary Hamming cube in \cite{kang2024bootstarp}.

While significant efforts have been dedicated to studying the critical probability of the process,  estimating percolation time, denoted by $T=\min \{t: A_t=V(G)\}$, has also garnered attention. In the deterministic setting, papers \cite{Michal}, \cite{Ivailo}, and \cite{Benevides} estimated the maximum percolation time for the process on the $n$-dimensional hypercube and the $[n]^2$ grid. In the probabilistic setting,  papers \cite{Bollobs2012TheTO}, \cite{Paul}, and \cite{Bollobs_bootstrap} estimated the distribution of percolation time on a $[n]^d$ grid.

In particular, from \cite{Paul} we have the following theorem.
\begin{theorem}  \label{theorem: Bollobas}
Let $ 0<p=p(n) <1$ be such that $\liminf p \log \log n > 2\lambda$ and 
$1-p=n^{-o(1)}$ (that is $\log 1/(1-p) \ll \log n$). Let  $T$ denotes the percolation time of a $p$-random subset of $[n]^2$ under the $2$-neighbor bootstrap percolation process and $\lambda=\frac{\pi^2}{18}$. Then we have 
$$T=\frac{(1+o(1))\log n}{ 2 \log (1/(1-p))},$$
with high probability as $n \rightarrow \infty$. 
\end{theorem}

In this chapter we investigate the distribution of the time of bootstrap percolation on the $d$-dimensional discrete torus $\mathbb{T}_n^d$. As our main result we extend Theorem \ref{theorem: Bollobas} in \cite{Paul} from the two-dimensional case to $d$-dimensions.

We highlight the primary challenges in extending Theorem~\ref{theorem: Bollobas} to higher dimensions, as well as our main contributions. In two dimensions \cite{Paul}, the key intuition is that 
the event in which a vertex $x$ remains uninfected at time $t$ is ``equivalent to'' the existence of an empty line segment of length roughly $t$ ``near'' $x$. In the higher dimensions, it is not immediately clear whether this key intuition still holds. However, we show that it remains valid in higher dimensions, provided that the infection threshold equals the dimension.

In the two-dimensional case, the idea of analyzing the intersection between a path of uninfected vertices and squares of a certain size was effective in reducing the combinatorial factor from counting the number of such paths. In higher dimensions, however, it is not obvious what geometric objects should be used to achieve a similar reduction. We find that replacing two-dimensional squares with their higher-dimensional analogs (cubes) is a viable approach.

Another challenge lies in generalizing the definition of the interior of a square from two dimensions to higher dimensions in a way that serves our purposes. In the two-dimensional setting, the notion of a square's interior was introduced to properly define certain events, and a similar concept is required in higher dimensions. We address these challenges and provide the necessary definitions to extend the analysis.

Our main result is given by the following theorem.

\begin{theorem} \label{theoremV1}
   Consider a graph $G$ as $G=\mathbb{T}_n^d$ and let $T$ be the percolation time. Assume that every vertex on the graph $G$ is initially infected with probability $p(n)$ independent of any other vertex and the infection threshold $r=d$. Assume that $ p(n) \geq \frac{C}{\log^{(d)}(n)}$ and $1-p=n^{-o(1)}$, where $C >0$ is sufficiently large and $\log^{(r)}(\cdot)$ denotes iterating the logarithm $r$ times. Then we have
   $$T= \Theta \left(\log_\frac{1}{1-p}n \right).$$
 with high probability as $n \rightarrow \infty$. 
\end{theorem}

 Let us explain our strategy in the proof of the lower bound of the percolation time. For the two-dimensional case, a natural example of an event that would prevent percolation from happening by time $t$ is the existence of an initially uninfected $[2t + 1] \times [2]$ rectangle. For the $d$-dimensional case, an event that there exists an initially uninfected $[2t+1] \times [2]^{d-1}$ rectangle initially would prevent percolation from happening by time $t$ by \cite{Bollobs2012TheTO}. One
can easily show that the largest $t$ for which such a rectangle is likely to exist is about
$\log_{\frac{1}{1-p}}n$. This observation essentially proves the lower bound of
Theorem \ref{theoremV1}. 

We briefly outline the key ideas behind the proof of the upper bound in Theorem~\ref{theoremV1}. The strategy is to show that if a vertex $x$ remains uninfected at time $t$, then this deterministically implies the existence of a "path" of nearby $[L]^d$ cubes, each of which is not internally spanned. By choosing $L$ appropriately (and hence controlling the length of the path), we can show that the probability of this event is small.

However, this summary is so brief and somewhat misleading. In fact, requiring each $[L]^d$ cube to be internally spanned is so strong that it is not necessary for bounding the probability of interest and, moreover, the probability of this stronger event differs significantly from the probability of interest. As a result, we must consider a weakened version of the condition, which introduces additional technical complexity. A more detailed sketch of the proof will be provided in Section~\ref{sect 3}.

Below by $C(d)$ we denote a constant which may depend on $d$. We will use $C(d)$ as generic notation, where the specific form of this function may be different in different expressions.

We write $[A_0]=\cup_{t=0}^{\infty} A_t$ and call $[A]$ the span of $A$. We say $A$ percolates on $G$ if $[A]=V(G)$ and a set $X \subset [n]^d$ is internally spanned if $X \subset [A \cup X]$.

A set of vertices is defined as empty if all its vertices are initially uninfected and call it non-empty otherwise.

\section{Sketch of the proof for the upper bound } \label{sect 3}

The upper bound in Theorem \ref{theoremV1} is roughly saying that the initially uninfected $[2t+1] \times [2]^{d-1}$ rectangles are the only obstacles to percolation by time $t$. 

Suppose a site $x$ in $\mathbb{T}_n^d$ is uninfected at time $t$. It is easy to see that at least $d+1$ neighbors of $x$ must be
uninfected at time $t-1$. Thus at least one of its neighbors among $(x_1+1,x_2,\cdots,x_d)$, $(x_1,x_2+1,\cdots,x_d)$, $\cdots$, $(x_1,x_2,\cdots,x_d+1)$ is uninfected at time $t-1$ and say it is $y$. Then at least one of $y$'s neighbors among $(y_1+1,y_2,\cdots,y_d)$, $(y_1,y_2+1,\cdots,y_d)$, $\cdots$, $(y_1,y_2,\cdots,y_d+1)$ is uninfected at time $t-2$. In fact, it is easy to see that there must exist a sequence of
$t$ initially uninfected vertices, starting with $x$, and
continuing on the direction parallel to the standard vector $e_1,e_2,\cdots, \text{or} \; e_d$ each time. 

We would like to show that by far the most probable way for this to occur is for this path to be aligned to form a line segment starting at $x$, or more specifically, we
would like to show that the probability that the uninfected paths exist is not much more than $(1-p)^{t}$ , which is the probability that a given line segment of length $t$ is initially uninfected.

One possible attempt of a proof might go as follows. Assume $x$ is uninfected at time $t$ and thus there is a path of uninfected vertices starting at $x$ and continuing on the direction parallel to $e_1,e_2,\cdots,e_d$ each time. This path, denoted by $P$, of uninfected vertices may not be a line segment and then this would cause a big combinatorial factors about the choice of such path.

However, it is possible to modify this idea. Following the ideas from \cite{Paul}, rather than counting paths along directions parallel to the standard vectors $e_1,e_2,\cdots, e_d$ of vertices individually, we look at the intersection of these paths with cubes $[L]^d$ and count these. First
we allow an initial time $t'=  o(t)$. By this time
we expect nearly all internally spanned cubes  $[L]^d$ to have been infected. Now consider just the first $t-t'$ vertices in the path $P$: at
time $t'$ they are still uninfected, and they intersect a path of cubes $[L]^d$ all of which are not internally spanned; we call such squares bad. Now we have an optimization question: how large should $L$ be to minimize the probability of this event 
that there is a path of bad cubes $[L]^d$ ? In order to have any hope,
the probability that a cube $[L]^d$ is bad should be at most $(1-p)^{(1+c)L}$ , for some $c > 0$.
This is because we would like to show that the probability there exists a path of bad cubes $[L]^d$ is about $(1 - p)^t$ , so we need the additional $c$ to overcome the combinatorial
factor coming from taking a union bound over all paths. Thus, $L$ must be large enough for the probability that a cube $[L]^d$ is bad to be small. Another reason $L$ should be large is to minimize the combinatorial factor. As $L$ increases, there are fewer paths of cubes $[L]^d$ inside a cube $[t-t']^{d}$, so the combinatorial factor decreases. On the other hand, $L$ cannot be too large, because the error time $t'=L^d$ must be $o(t)$.

This second attempt of a proof is also not quite right: the probability that a cube $[L]^d$ is bad, as we have defined it, is at least $(1 - p)^L$ because if any of the $2^{d-1}d$ ``sides'' of
the cube is empty then the cube cannot be internally spanned. On the other hand we have said that the probability needs to be at most $(1-p)^{(1+c)L}$, so our definition of bad cannot be the right one. The way around this is as follows. Let us define a cube $[L]^d$ is bad if it is not internally spanned except for its ``sides'', which will be defined precisely in Section \ref{Upper bound} . It turns out this definition is the right one to use for our purpose.

\section{Lower Bound}

To prove the lower bound for the percolation time $T$ we will use a theorem by Bollob{\'a}s, Smith and Uzzell \cite{Bollobs2012TheTO}.

First let us observe that the states of vertices at $l_1$ distance greater than $t$ from the vertex $x$ cannot affect whether or not the vertex $x$ is infected at time $t$. Therefore we can restrict ourselves to the $l_1$-ball
$$B_t(x)=\{y \in [n]^d:|x-y|_{l_1} \leq t \}.$$

The following theorem from \cite{Bollobs2012TheTO} will be used in the proof.
\begin{theorem} \label{theoremV2} 
   
   Let  $$\text{ex}_{d,r}(t):= \min \{|B_t(0) \backslash A_0|: 0 \notin A_t   \}.$$ and
   $$P_{d,r}(t):=\{x \in B_t(0):x_{d-r+2},\dots,x_d \in \{0,1\}\}.$$ where $x=(x_1,x_2,\dots,x_d)$.
   
   Then for every $ 2 \leq  r \leq d$
   $$\text{ex}_{d,r}(t)=|P_{d,r}(t)|.$$
\end{theorem}
   
Using this result, let us prove the lower bound for the percolation time. Assume that $d=r$.

\begin{proof}
    
Let $E_1$ be the event that every vertex in $\{x \in B_t(0): x_2,...,x_d \in \{0,1\} \}$ is initially uninfected and $E_2$ be the event that the vertex $(0,0,...,0)$ is uninfected at time $t$.

From Theorem~\ref{theoremV2}, the event $E_1$ implies the event $E_2$. Therefore, the event $E_1$ implies the event $\{T >t \}$.  

Let $E_3$ be the event that every rectangle $[2t+1] \times [2]^{d-1}$ is nonempty. 
It is easy to see that the event $\{T \leq t \}$ implies the event $E_3.$ Let us divide $[n]^d$ into $\frac{n^d}{2^dt}$ disjoint rectangles $[2]^{d-1} \times [2t+1]$. Let $E_4$ be an event that every such disjoint rectangle $[2]^{d-1} \times [2t+1]$ is nonempty. It is easy to see that the event $E_3$ implies the event $E_4$. 

Therefore, we have 
\begin{align*}
   \mathbb{P}(T \leq t) & \leq \left(1-\left(1-p \right)^{2^dt} \right)^{\frac{n^d}{2^dt}}\\
   & \leq \exp \left(-\left(1-p \right)^{2^dt}\frac{n^d}{2^dt} \right)\\
   & = \exp\left(-\exp\left(d\log(n)-\log(2^dt)-2^dt\log\left(\frac{1}{1-p} \right) \right) \right)\\
   & = o(1)
\end{align*}
if $$\frac{2^dt \log(\frac{1}{1-p})}{d\log(n)} \leq 1.$$
Thus, with high probability, 
$$T \geq C(d) \frac{\log (n)}{\log(\frac{1}{1-p})}.\qedhere $$ 
\end{proof}

\section{Upper Bound} \label{Upper bound}

Before proceeding, we introduce some elements of notation and definitions.

Define a subcube $$[(a_1,b_1),(a_2,b_2),...,(a_d,b_d)]:=\{(x_1,x_2,...,x_d) \in [n]^d: a_i \leq x_i \leq b_i  \; \forall i \in [d]\}$$

Define a sub-grid
\begin{align*}
    [(a_1,b_1),(a_2,b_2),...,(a_{j-1},b_{j-1}),(a_j),(a_{j+1},b_{j+1}),...,(a_d,b_d)]:=\\\{(x_1,x_2,...,x_d) \in [n]^d: 
     x_j=a_j, a_i \leq x_i \leq b_i \; \forall i \in [d] \backslash \{j\}\}
\end{align*}

Define a line segment 
\begin{align*}
    [(a_1),(a_2),\cdots,(a_{j-1}),(a_j,b_j),(a_{j+1})...,(a_d)]:=\{(x_1,x_2,...,x_d) \in [n]^d: \\
     x_i=a_i, a_j \leq x_j \leq b_j \; \forall i \in [d] \backslash \{j\}\}
\end{align*}

Define a side of the cube $[m]^d$ as of the form $[(a_1),(a_2),...,(a_{j-1}),(1,m),(a_{j+1}),...,(a_d)]$ with $a_i \in \{1,m \}$ for all $i \in \{[d] \backslash j \}$ for some $j \in [d].$ Note that there are in total $2^{d-1}d$ sides of $[m]^d$.



Let $\mathfrak{S}_d$ be the symmetric group on $d$ coordinates. Then define:
$$\text{Perm}_d(x):=\{\sigma x: \sigma \in \mathfrak{S}_d\},$$
where for $x = [(x_1), \dots, (x_d)]$, we define:
$$\sigma x:=[\left(x_{\sigma(1)}),\cdots,(x_{\sigma(d)} \right)].$$
Note that by abuse of the notation if $x_i$ is an interval for some $i \in [d]$, regarding the permutation it is treated as a single value. 

We will give the definition of interior($[m]^d$) such that it serves our purpose.

Define the interior of $[m]^d$ as 
$$\text{int}([m]^d):=[m]^d \backslash \{\text{Perm}_d [\left((a_1),(a_2),...,(a_{j-1}),(1,m),(a_{j+1}),...,(a_d) \right)]: a_i \in \{1,m\} \; \text{for} \; i \in [d] \backslash \{j\} \}$$ i.e., int($[m]^d$) includes every vertex in $[m]^d$ except for its sides.

Let $A \sim \text{Bin}([m]^d,p)$ be the initial set. Define the cube $D=[m]^d$ to be good if its span contains its interior. Formally, the cube $D$ is good if $\text{int}(D) \subset [D \cap A]$. The cube $D$ is strongly good if it is internally spanned, i.e., $D \subset [D \cap A]$. At last, $D$ is semi-good if it is good but not strongly good, and bad if it is not good. 


We denote by $\eta_{m,r}$ the probability that $[m]^d$ is bad with the infection threshold $r$ and in order to simplify the notation, sometimes we may drop the dependence on $r$ if it is clear on the context.

We need the following theorem by Bollob{\'a}s, Balogh, Duminil-Copin, and Morris \cite{Balogh}.

\begin{theorem} \label{theoremV3} 
   For every $d \geq r \geq 2$, there exists an explicit constant $\lambda(d,r) > 0$ such that the critical probability for $[n]^d$ with infection threshold $r$,denoted by $p_c([n]^d,r)$, can be expressed as follows: 
   $$
   p_c([n]^d,r)=\left[ \frac{\lambda(d,r)+o(1)}{\log^{(r-1)} n} \right]^{d-r+1}.$$
\end{theorem}

\subsection{$d=r=3$}

Before stating the next lemma let us state one fact which will be used very often in this section. 
\begin{fact} \label{fact}
Assume that $\mathbb{P}(B)\leq a $, $\mathbb{P}(C) \leq b$, $\mathbb{P}(A \cap B)=0$, and  $\mathbb{P}(A \cap C)=0 $. Moreover, assume that the event B and the event C are independent. Then 
$$\mathbb{P}(A) \leq 2ab.$$
\end{fact}
\begin{proof}
\begin{align*}
    \mathbb{P}(A) &= \mathbb{P}(A \cap B \cap C) + \mathbb{P}(A \cap \left(B \cap C \right)^c)\\
    & = \mathbb{P}(A \cap \left(B \cap C \right)^c)\\
    & \leq \mathbb{P}(A \cap B^c) +\mathbb{P}(A \cap C^c)\\
    & \leq \mathbb{P}(A \cap B^c \cap C) + \mathbb{P}(A \cap B^c \cap C^c) +  \mathbb{P}(A \cap C^c \cap B) + \mathbb{P}(A \cap C^c \cap B^c)\\
    & \leq 2 \mathbb{P}(B^c \cap C^c)\\
    & = 2ab. \qedhere
\end{align*}

\end{proof}

When $d=r=3$, we use Lemma ~\ref{lemmaV1} to derive a recursive relation on the probability that $[m]^3$ is bad.

\begin{lemma} \label{lemmaV1} 
Let $\eta_{m,3}$ be the probability that $[m]^3$ is bad with the infection threshold $3$. We have 
   $$\eta_{2m,3} \leq C(d)\eta_{m,3}^3+C(d)m^2(1-p)^{4m-8}.$$
\end{lemma}

The intuition behind the proof of Lemma ~\ref{lemmaV1} is as follows. We begin by partitioning the cube $[2m]^3$ into eight disjoint subcubes and break the proof into several subcases according to the number of bad subcubes in this 8-tuple. Specifically, we analyze the cases where exactly one or two of the eight subcubes are bad as well as the case where none are bad.  The main technical part of the proof is devoted to the analysis to these three subcases, where we study the interaction arising among the 8 subcubes. Further the probability of having at least three bad subcubes can be estimated by a very rough upper bound of $C\eta_{m}^3$, as established in Lemma ~\ref{lemmaV1}.

Before proving this lemma we need to introduce notation and prove several auxiliary lemmas. 

Let us partition the cube $[2m]^3$ into 8 disjoint subcubes as shown in Figure ~\ref{figure1}.

$$C_1=[(1,m),(1,m),(1,m)], C_2=[(m+1,2m),(1,m),(1,m)],$$ 
$$C_3=[(1,m),(m+1,2m),(1,m)],C_4=[(m+1,2m),(m+1,2m),(1,m)],$$ 
$$C_5=[(1,m),(1,m),(m+1,2m)], C_6=[(m+1,2m),(1,m),(m+1,2m)],$$ 
$$C_7=[(1,m),(m+1,2m),(m+1,2m)],C_8=[(m+1,2m),(m+1,2m),(m+1,2m)].$$
\begin{figure}[htbp]
\centering
\begin{tikzpicture}
		[cube/.style={very thick,black},
			grid/.style={very thin,gray},
			axis/.style={->,blue,thick}]
   \draw[axis] (0,0,0) -- (6,0,0) node[anchor=west]{$x$};
	\draw[axis] (0,0,0) -- (0,6,0) node[anchor=west]{$z$};
	\draw[axis] (0,0,0) -- (0,0,6) node[anchor=west]{$y$};

    \draw[cube] (0,2.5,0) -- (0,4.5,0) -- (2,4.5,0) -- (2,2.5,0) -- cycle;
    \draw[cube] (0,2.5,2) -- (0,4.5,2) -- (2,4.5,2) -- (2,2.5,2) -- cycle;
        \draw[cube] (0,2.5,0) -- (0,2.5,2);
	\draw[cube] (0,4.5,0) -- (0,4.5,2);
	\draw[cube] (2,2.5,0) -- (2,2.5,2);
	\draw[cube] (2,4.5,0) -- (2,4.5,2);

  \draw[cube] (0,2.5,2.5) -- (0,4.5,2.5) -- (2,4.5,2.5) -- (2,2.5,2.5) -- cycle;
   \draw[cube] (0,2.5,4.5) -- (0,4.5,4.5) -- (2,4.5,4.5) -- (2,2.5,4.5) -- cycle;
         \draw[cube] (0,2.5,2.5) -- (0,2.5,4.5);
	\draw[cube] (0,4.5,2.5) -- (0,4.5,4.5);
	\draw[cube] (2,2.5,2.5) -- (2,2.5,4.5);
	\draw[cube] (2,4.5,2.5) -- (2,4.5,4.5);

 \draw[cube] (2.5,2.5,0) -- (2.5,4.5,0) -- (4.5,4.5,0) -- (4.5,2.5,0) -- cycle;
    \draw[cube] (2.5,2.5,2) -- (2.5,4.5,2) -- (4.5,4.5,2) -- (4.5,2.5,2) -- cycle;
        \draw[cube] (2.5,2.5,0) -- (2.5,2.5,2);
	\draw[cube] (2.5,4.5,0) -- (2.5,4.5,2);
	\draw[cube] (4.5,2.5,0) -- (4.5,2.5,2);
	\draw[cube] (4.5,4.5,0) -- (4.5,4.5,2);

 \draw[cube] (2.5,2.5,2.5) -- (2.5,4.5,2.5) -- (4.5,4.5,2.5) -- (4.5,2.5,2.5) -- cycle;
    \draw[cube] (2.5,2.5,4.5) -- (2.5,4.5,4.5) -- (4.5,4.5,4.5) -- (4.5,2.5,4.5) -- cycle;
        \draw[cube] (2.5,2.5,2.5) -- (2.5,2.5,4.5);
	\draw[cube] (2.5,4.5,2.5) -- (2.5,4.5,4.5);
	\draw[cube] (4.5,2.5,2.5) -- (4.5,2.5,4.5);
	\draw[cube] (4.5,4.5,2.5) -- (4.5,4.5,4.5);
	\draw[cube] (0,0,0) -- (0,2,0) -- (2,2,0) -- (2,0,0) -- cycle;
	\draw[cube] (0,0,2) -- (0,2,2) -- (2,2,2) -- (2,0,2) -- cycle;
 
	\draw[cube] (2.5,0,0) -- (2.5,2,0) -- (4.5,2,0) -- (4.5,0,0) -- cycle;
	\draw[cube] (2.5,0,2) -- (2.5,2,2) -- (4.5,2,2) -- (4.5,0,2) -- cycle;


       \draw[cube] (0,0,2.5) -- (0,2,2.5) -- (2,2,2.5) -- (2,0,2.5) -- cycle;
	\draw[cube] (0,0,4.5) -- (0,2,4.5) -- (2,2,4.5) -- (2,0,4.5) -- cycle;


        \draw[cube] (2.5,0,2.5) -- (2.5,2,2.5) -- (4.5,2,2.5) -- (4.5,0,2.5) -- cycle;
	\draw[cube] (2.5,0,4.5) -- (2.5,2,4.5) -- (4.5,2,4.5) -- (4.5,0,4.5) -- cycle;
 
	\draw[cube] (0,0,0) -- (0,0,2);
	\draw[cube] (0,2,0) -- (0,2,2);
	\draw[cube] (2,0,0) -- (2,0,2);
	\draw[cube] (2,2,0) -- (2,2,2);
 
    \draw[cube] (2.5,0,0) -- (2.5,0,2);
	\draw[cube] (2.5,2,0) -- (2.5,2,2);
	\draw[cube] (4.5,0,0) -- (4.5,0,2);
	\draw[cube] (4.5,2,0) -- (4.5,2,2);


        \draw[cube] (0,0,2.5) -- (0,0,4.5);
	\draw[cube] (0,2,2.5) -- (0,2,4.5);
	\draw[cube] (2,0,2.5) -- (2,0,4.5);
	\draw[cube] (2,2,2.5) -- (2,2,4.5);


\draw[cube] (2.5,0,2.5) -- (2.5,0,4.5);
	\draw[cube] (2.5,2,2.5) -- (2.5,2,4.5);
	\draw[cube] (4.5,0,2.5) -- (4.5,0,4.5);
	\draw[cube] (4.5,2,2.5) -- (4.5,2,4.5);

\draw[cube] (10,2,2.5) -- (10,2,4.5);
\draw[cube] (10,2,2.5)-- (12,2,2.5);
\draw[cube] (12,2,2.5)-- (12,2,4.5);
\draw[cube] (10,2,4.5)-- (12,2,4.5);

\draw[cube] (10,4,2.5) -- (10,4,4.5);
\draw[cube] (10,4,2.5)-- (12,4,2.5);
\draw[cube] (12,4,2.5)-- (12,4,4.5);
\draw[cube] (10,4,4.5)-- (12,4,4.5);
  \coordinate[label=left:$C_1$]  (C1)  at (10,2,2.5);
  \coordinate[label=left:$C_3$] (C3)  at (10,2,4.5);
  \coordinate[label=right:$C_2$]  (C2)  at (12,2,2.5);
  \coordinate[label=right:$C_4$] (C4)  at (12,2,4.5);

  \coordinate[label=left:$C_5$]  (C5)  at (10,4,2.5);
  \coordinate[label=left:$C_7$] (C7)  at (10,4,4.5);
  \coordinate[label=right:$C_6$]  (C6)  at (12,4,2.5);
  \coordinate[label=right:$C_8$] (C8)  at (12,4,4.5);

\end{tikzpicture}

\caption{Subcubes $C_1,C_2,\cdots,C_8$}
\label{figure1}
\end{figure}
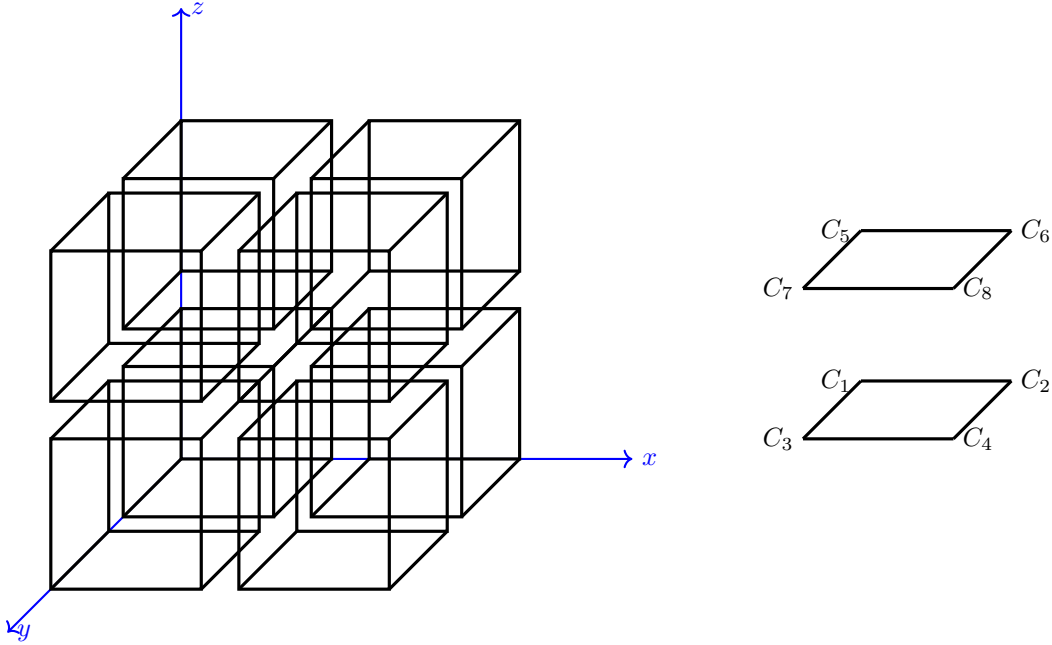

We will first define a few events. 
\begin{itemize}
    
\item $E:=\{[2m]^3$ is good \}.

\item $A$:=\{all cubes $C_1,C_2,\cdots, C_8$ are good\}.

\item $A_i$:=\{exactly $i$ subcubes among $C_1$,$C_2$,...,$C_8$ are bad\} for $i \in \{1,\dots,8\}$.
\end{itemize}

Note that 
\begin{equation} \label{eq: P}
    \mathbb{P}(E^c) = \mathbb{P}(E^c \cap A)+\mathbb{P}(E^c \cap A^c).
\end{equation}

Let us consider the term $\mathbb{P}(E^c \cap A)$ in (\ref{eq: P}) and this is the case where all 8 subcubes $C_1,\dots,C_8$ are good. Since the sides of these 8 subcubes may still remain uninfected the event $A$ does not necessarily imply that the cube $[2m]^d$ is good. Therefore, we have to analyze the interaction between the events that some vertices in the sides of $C_1,\dots,C_8$ are initially infected and the event $A$. Moreover, in order for our method to work the probability that the complement of the events that some vertices in the edges of $C_1,\dots,C_8$ are initially infected needs to be small and in fact it needs to be $\leq (1-p)^{4m}$. Let us define these events precisely. 

Let $B(m,(1,2m),a_3)$ for $a_3 \in \{1,m+1,2m\}$ denote the event that at least one of the following 2 line segments is non-empty:
\begin{align*}
[(m),(1,2m),(a_3)],\text{or} \quad [(m+1),(1,2m),(a_3)].
\end{align*}

Note that the above 2 line segments belong to the sides of $C_5,\dots,C_8$ if $a_3=2m$, and belong to the sides of $C_1,\dots,C_4$ if $a_3=1$.


$B(\pi (m),\pi((1,2m)),\pi (a_3))$ are defined similarly for all $\pi$ where $\pi$ is an ordering of the set $\{m,(1,2m),a_3\}$.


Let $B(m,m,(c_1,c_2))$ denote the event that at least one of the following 4 line segments is nonempty:
\begin{align*}
& \{[(m),(m),(c_1,c_2)],[(m),(m+1),(c_1,c_2)],\\
& [(m+1),(m),(c_1,c_2)],[(m+1),(m+1),(c_1,c_2)] \}.
\end{align*}

$B(\pi (m),\pi(m),\pi ((c_1,c_2)))$ are defined similarly for all $\pi$ where $\pi$ is an ordering of the set $\{m,(m),(c_1,c_2)\}$.

Let us define an event $B$ such that if the events $A$ and $B$ occur then the event $E$ occurs and $\mathbb{P}(B^c) \leq C (1-p)^{4m}$ with $C>0$. 

Let 
$$B_1=\cap_{a_3 \in \{2m,m+1,m,1\}}B(m,(1,2m),a_3),$$ 
$$B_2=\cap_{a_3 \in \{2m,m+1,m,1\}}B((1,2m),m,a_3),$$
$$B_3=B(m,m,(m+1,2m) \cap B(m,m,(1,m)),$$
$$B_4=B(m,1,(1,2m)) \cap B((1,2m),1,m),$$
$$B_5=B(2m,m,(1,2m)) \cap B(2m,(1,2m),m),$$
$$B_6=B((1,2m),2m,m) \cap B(m,2m,(1,2m)),$$
$$B_7=B(1,m,(1,2m)) \cap B(1,(1,2m),m).$$

Let $B=\cap_{i \in [7]}B_i$. 
Note that 
\begin{equation} \label{eq: Pp}
    \mathbb{P}(E^c) = \mathbb{P}(E^c \cap A \cap B)+\mathbb{P}(E^c \cap A^c \cap B)+\mathbb{P}(E^c \cap A \cap B^c)+\mathbb{P}(E^c \cap A^c \cap B^c).
\end{equation}

Therefore we will try to bound the 4 terms on the right side of (\ref{eq: Pp}). If we can show that $A \cap B= E$, then the first term on the right side of (\ref{eq: Pp}) vanishes. The next lemma shows this. 

\begin{lemma} \label{lemma: A B}
If the events $A$ and $B$ occur, then the event $E$ occurs, i.e., $[2m]^3$ is good. 
\end{lemma}
\begin{proof}
    Note that since all 8 subcubes $C_1,C_2,...,C_8$ are good, the interior of $C_i$ has been infected after $m^3$ steps for all $i \in [8]$. Thus after $m^3$ steps the only possibly uninfected vertices are the vertices on the sides for each cube $C_i$. Without loss of generality, since the event $B(m,(1,2m),2m)$ occurs, we can assume that there is a vertex $a=(a_1,a_2,a_3)$ in $[(m),(1,m),(2m)]$ is initially infected.

    Let us consider the vertices in $[(1,2m),(1,2m),(2m)]$ after $m^3$ steps. Since all 8 subcubes $C_1,C_2,\cdots,C_8$ are good, every vertex in $[(m),(2,m-1),2m]$ has at least 2 infected neighbors after $m^3$ steps. Since the vertex $a$ is initially infected, it takes at most $m^3+m$ steps to infect every vertex in $[(m),(2,m-1),2m]$. 
    
    Since the vertex $a$ is initially infected and all 8 subcubes are good, the vertex $y=(a_1+1,a_2,a_3)$ has 3 infected neighbors after $m^3$ so it will be infected in the next step and then it takes at most $m^3+m$ steps to infect every vertex in $[(m+1),(2,m-1),2m]$.

    Note that every vertex in $$[(m),(m),(m+2,2m-1)],$$ $$[(m+1),(m),(m+2,2m-1)],$$ $$[(m),(m+1),(m+2,2m-1)]$$ and $$[(m+1),(m+1),(m+2,2m-1)]$$ already has 2 infected neighbors since all 8 small cubes are good. Therefore since the event $B(m,m,(m+1,2m))$ occurs, it takes at most $m^3+m$ steps to infect every vertex in $$[(m),(m),(m+2,2m-1)],$$ $$[(m+1),(m),(m+2,2m-1)],$$ $$[(m),(m+1),(m+2,2m-1)],$$ and $$[(m+1),(m+1),(m+2,2m-1)].$$ Then note that after $m^3+m$ steps, the vertices $(m,m,2m)$ and $(m+1,m,2m)$ have at least two infected neighbors and the vertices $(m,m+1,2m)$ and $(m+1,m+1,2m)$ have at least one infected neighbor.  
    
    Since the event $B((1,2m),m,2m)$ occurs, a vertex $b$ in $$[(1,m),(m),(2m)],$$ $$[(1,m),(m+1),(2m)],$$ $$[(m+1,2m),(m),(2m)],$$ or $$[(m+1,2m),(m+1),(2m)]$$ is initially infected. Then following the same logic, every vertex in $[(2,m-1),(m),2m]$ and $[(2,m-1),(m+1),(2m)]$ will be infected after at most $m^3+m$ steps. Thus the vertex $(m,m,2m)$ has 3 infected neighbors after at most $m^3+m$ steps and will be infected in the next step and after that the vertices $(m,m+1,2m)$ and $(m+1,m,2m)$ will be infected. Since the vertex $(m+1,m,2m)$ has been infected after at most $m^3+m$ steps it is easy to see that every vertex in  $[(m+1,2m-1),(m),(2m)] $ will be infected after at most $m^3+2m$ steps. Then every vertex in $[(m+1,2m-1),(m+1),(2m)]$ has 3 infected neighbors after at most $m^3+2m$ steps and will be infected in the next step. Similarly, every vertex in $$[(m),(m+1,2m-1),(2m),] \quad [(m+1),(m+1,2m-1),(2m)$$ will be infected after at most $m^3+3m$ steps.

    The rest of the uninfected vertices on the edges of $C_1,\dots,C_8$ can be handled in a very similar way and so we have the desired result. \qedhere
    
    \end{proof}
Using Lemma \ref{lemma: A B}, we have 
\begin{align*}
  \mathbb{P}(E^c) &= \mathbb{P}(E^c \cap A \cap B)+\mathbb{P}(E^c \cap A^c \cap B)+\mathbb{P}(E^c \cap A \cap B^c)+\mathbb{P}(E^c \cap A^c \cap B^c) \\
  &\leq \mathbb{P}(E^c \cap A^c\cap B)+2 \mathbb{P}(B^c)\\
  & = 2 \mathbb{P}(B^c) + \mathbb{P}(\cup_{i=1}^8(E^c \cap A_i\cap B)) \\
  & \leq C' (1-p)^{4m}+\mathbb{P}(E^c \cap A_1 \cap B) + \mathbb{P}(E^c \cap A_2 \cap B) + \sum_{i=3}^8 \mathbb{P}(A_i) \\
  & \leq C' (1-p)^{4m}+\mathbb{P}(E^c \cap A_1 \cap B) + \mathbb{P}(E^c \cap A_2 \cap B) + \sum_{i=3}^8 \binom{8}{i} \eta_{m,3}^i \\
  & \leq C' (1-p)^{4m}+\mathbb{P}(E^c \cap A_1 \cap B) + \mathbb{P}(E^c \cap A_2 \cap B)+ C \eta_{m,3}^3
\end{align*}
where $C>0$ and $C'>0$ are absolute constants.
 Now let us estimate the term $\mathbb{P}(E^c \cap A_1 \cap B)$. 
 
Before proceeding, let us define a few events. 
$$B_8=\cap_{a_2 \in \{2m,m+1\}}B(m,a_2,(1,2m)) \cap B((1,2m),a_2,m) \cap B(m,(m+1,2m),m),$$
$$B_9=\cap_{a_1 \in \{2m,m+1\}}B(a_1,m,(1,2m)) \cap B(a_1,(1,2m),m) \cap B((m+1,2m),m,m),$$
$$B_{10}=\cap_{a_3 \in \{2m,m+1\}}B(m,(1,2m),a_3) \cap B((1,2m),m,a_3) \cap B(m,m,(m+1,2m)).$$

\begin{lemma} \label{lemma: E A1 B}
The probability that $[2m]^3$ is bad along with the occurrence of the event that there is exactly one bad subcube is correlated with double empty line segments of length $2m$. More precisely, we have 
 $$\mathbb{P}(E^c \cap A_1) \leq C (1-p)^{4m}.$$
\end{lemma}
\begin{proof}



Due to symmetry, without loss of generality, we may assume that the subcube $C_1$ is bad. The main idea of the proof is as follows. We will construct 4 set of events $S_{start},S_{x=1},S_{y=1}$ and $S_{z=1}$ such that, if the events in these 4 sets occur along with events $A_1$ and $B$, then the event $E$ also occurs; that is, the cube $[2m]^d$ is good. Moreover, the probability that some of these events in these 4 sets fails to occur is approximately $(1-p)^{4m}$. If the probability that any of the events in these four sets fails to occur is significantly larger than $(1-p)^{4m}$, then a more detailed analysis is required. The rough idea is that other events, not contained in $S_{\text{start}}$, $S_{x=1}$, $S_{y=1}$, or $S_{z=1}$, can also lead to percolation. The probability of the complement of these additional events, together with the complement of the events in $S_{\text{start}}$, $S_{x=1}$, $S_{y=1}$, and $S_{z=1}$, is  $\leq (1-p)^{4m}$. The exact details are presented below.

\vspace{2mm}
 The infection spreads as follows. First, after at most $m^3$ steps, the interior of $C_2,\dots,C_8$ have been infected since $C_2,\dots, C_8$ are good. Then the sides of the subcubes $C_2, \dots, C_8$ become infected (except for those on $[(1,2m),(1),(1,2m)]$, $[(1),(1,2m),(1,2m)]$ and $[(1,2m),(1,2m),(1)]$ and those belonging to the sides of $[2m]^3$ )after at most $m^3+Cm$ steps, due to the occurrence of the events in $S_{start}$. Then, after at most $2m^3+Cm$ steps the interior of $C_1$ have been infected. Finally, the uninfected vertices on $[(1,2m),(1),(1,2m)]$, $[(1),(1,2m),(1,2m)]$ and $[(1,2m),(1,2m),(1)]$ will be infected after at  ost $2m^3+Cm'$ steps, due to the occurrence of the event in $S_{x=1},S_{y=1}$ and $S_{z=1}$.

Now we will start defining the events in the set $S_{start}$. It turns out that the following $B_8,B_9$, and $B_{10}$ serve our purpose.

\vspace{2mm}

Now we can focus on analyzing how to infect the vertices on $[(1,2m),(1),(1,2m)], [1,(1,2m),(1,2m)]$, and $[(1,2m),(1,2m),(1)].$ Due to symmetry we only need to consider infecting vertices on $[(1,2m),(1),(1,2m)]$.

Consider the vertices in the region $[(1,2m),(1),(1,2m)]$ as shown in Figure~\ref{figure5}, where the vertices in the shaded area are infected, and those in the white area remain uninfected. Note that after $2m^3+Cm$ steps among those vertices the only possible uninfected vertices are on $[(1,m),(1),(1,m)]$ and the sides of $C_2,C_5, C_6$.  

\begin{figure}[htbp]
\centering
\begin{tikzpicture}
		[cube/.style={very thick,black},
			grid/.style={very thin,gray},
			axis/.style={->,blue,thick}]

			
	\draw[axis] (0,0,0) -- (5,0,0) node[anchor=west]{$x$};
	\draw[axis] (0,0,0) -- (0,5,0) node[anchor=west]{$z$};
	\draw[axis] (0,0,0) -- (0,0,5) node[anchor=west]{$y$};

	\draw[cube] (0,0,0) -- (0,2,0) -- (2,2,0) -- (2,0,0) -- cycle;
	\draw[cube] (0,2.5,0) -- (0,4.5,0) -- (2,4.5,0) -- (2,2.5,0) -- cycle;
       \draw[cube] (2.5,0,0) -- (2.5,2,0) -- (4.5,2,0) -- (4.5,0,0) -- cycle;
       \draw[cube] (2.5,2.5,0) -- (2.5,4.5,0) -- (4.5,4.5,0) -- (4.5,2.5,0) -- cycle;

  \fill[gray] (0,2.5,0) -- (0,4.5,0) -- (2,4.5,0) -- (2,2.5,0) -- cycle;
  \fill[gray] (2.5,0,0) -- (2.5,2,0) -- (4.5,2,0) -- (4.5,0,0) -- cycle;
  \fill[gray] (2.5,2.5,0) -- (2.5,4.5,0) -- (4.5,4.5,0) -- (4.5,2.5,0) -- cycle;
  
\end{tikzpicture}
\caption{Configuration on $[(1,2m),(1),(1,2m)]$ after $2m^3+Cm$ steps.} 
\label{figure5}
\end{figure}
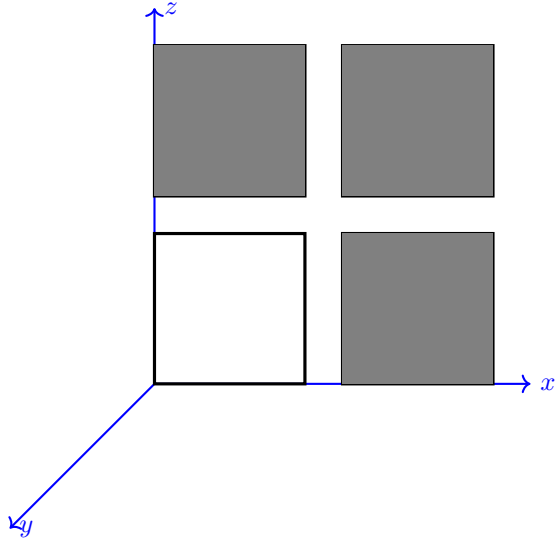

Let us define the event $D(y=1)$ in $S_{y=1}$ as 
$$D(y=1):= \{[(1,m),(1),(m)] \; \text{is non-empty} \} \cap \{ [(m),(1),(1,m)] \; \text{is non-empty} \}. $$

It is easy to see that if the events $A_1, B, B_i $ for $i \in \{8,9,10\}$ and $D(y=1)$ occurs, then every vertex on $[(1,2m),(1),(1,2m)]$ will be infected except for those on the sides of $[2m]^3$. 

We will define the event $D(x=1)$ in $S_{x=1}$  and the event $D(z=1)$ in $S_{z=1}$ accordingly.

Therefore, we have 
$$\mathbb{P}(E^c \cap A_1 \cap B \cap (\cap_{i=8}^{10}B_i) \cap D(x=1) \cap D(y=1) \cap D(z=1) )=0.$$
and thus we have 
\begin{align*}
  \mathbb{P}(E^c \cap A_1 \cap B) & =  \mathbb{P}(E^c \cap A_1 \cap B \cap (\cap_{i=8}^{10}B_i) \cap D(x=1) \cap D(y=1) \cap D(z=1))\\
  &+\mathbb{P}(E^c \cap A_1 \cap B \cap ((\cap_{i=8}^{10} B_i) \cap D(x=1) \cap D(y=1) \cap D(z=1))^c)\\
  & =\mathbb{P}(E^c \cap A_1 \cap B \cap ((\cap_{i=8}^{10} B_i) \cap D(x=1) \cap D(y=1) \cap D(z=1))^c)
\end{align*} 

We need to further analyze the event $E^c \cap A_1 \cap B \cap ((\cap_{i=8}^{10} B_i) \cap D(x=1) \cap D(y=1) \cap D(z=1))^c$. Note that $\mathbb{P}(B_i^c) \leq C(1-p)^{4m}$ for $i \in \{8,9,10\}$.

Since $\mathbb{P}(D(y=1)^c)=2(1-p)^{m}$, we need to construct some other events. Let us consider the event $D(y=1)^c$ and we have 

\begin{align*}
D(y=1)^c = D_1 \cup D_2 \cup D_3
\end{align*}
where 
$$D_1=\{[(1,m),(1),(m)] \; \text{is non-empty}\} \cap \{ [(m),(1),(1,m)] \; \text{is empty} \}, $$
$$D_2=\{[(1,m),(1),(m)] \; \text{is empty} \} \cap \{ [(m),(1),(1,m)] \; \text{is non-empty} \},  $$
and 
$$ D_3=\{[(1,m),(1),(m)] \; \text{is empty} \} \cap \{ [(m),(1),(1,m)] \; \text{is empty} \}. $$

We will need to do further analysis based on the event $D(y=1)^c$. 

Case 1: Assume that the event $D_1$ occurs.

In order to proceed we need to define another event. Let $F_1(y=1)$ be the event that one of $[(m),(1),(m+1,2m)]$ or $[(m+1),(1),(m+1,2m)]$ is non-empty. It is easy to see that if the events $A_1, B, D_1, F_1(y=1)$ and $B_i$ for $i \in \{8,9,10\}$ occur, then every vertex on $[(1,2m),(1),(1,2m)]$ will be infected after at most $2m^2+Cm$ steps, except for those on the sides of $[2m]^3$. 

Now we need to consider $A_1 \cap B \cap_{i=8}^{10} B_i \cap D_1 \cap F_1(y=1)^c$. Let $F_2(y=1)$ be the event that $[(m+1),(1),(1,m)]$ is non-empty. It is easy to see that if the events $A_1, B,  D_1, F_1^c(y=1),F_2(y=1), B_i$ for $i \in \{8,9,10\}$ occur, then every vertex on $[(1,2m),(1),(1,2m)]$ will be infected except for the sides of $[2m]^3$.

Since $\mathbb{P}(D_1 \cap F_1(y=1)^c \cap F_2(y=1)^c) \leq (1-p)^{4m}$ and Fact $\ref{fact}$, we have the desired result. 

\vspace{2mm}
Case 2: Assume that the event $D_2$ happens.

The analysis is the same as Case 1 due to symmetry.

\vspace{2mm}
Case 3: Assume that the event $D_3 $ happens. 


Before proceeding we need to define another event. Define $F_3(y=1)$ to be the event that both $[(m+1),(1),(m+1,2m)]$ and $[(m+1,2m),(1),(m+1)]$ are non-empty. It is easy to see that if the events $A_1,B,D_3$, $F_3(y=1)$ and $B_i$ for $i \in \{8,9,10\}$ occur, then every vertex on $[(1,2m),(1),(1,2m)]$ have been infected after at most $2m^3+Cm$ steps, except for those on the sides of $[2m]^3$.

Now we need to analyze $F_3(y=1)^c$. Define the event $H$ be the event that $[(1,m),(1),(m+1)]$ is non-empty. 

We have 
$$F_3(y=1)^c=F_1 \cup F_2 \cap F_3$$
where 
$$F_1=\{[(m+1),(1),(m+1,2m)] \; \text{is non-empty}\} \cap \{ [(m+1,2m),(1),(m+1)] \; \text{is empty} \}, $$
$$F_2=\{[(m+1),(1),(m+1,2m)] \; \text{is empty} \} \cap \{ [(m+1,2m),(1),(m+1)] \; \text{is non-empty} \},  $$
and 
$$ F_3=\{[(m+1),(1),(m+1,2m)] \; \text{is empty} \} \cap \{ [(m+1,2m),(1),(m+1)] \; \text{is empty} \}. $$ 
Assume that $F_1$ occurs. Then again it is easy to see that if the events $A_1,B,D_3$, $F_1$, $H$ and $B_i$ for $i \in \{8,9,10\}$ occur, then every vertex on $[(1,2m),(1),(1,2m)]$ have been infected after at most $2m^3+Cm$ steps, except for those on the sides of $[2m]^3$. 
Note that $\mathbb{P}(D_3 \cap F_1 \cap H^c)=(1-p)^{4m}$.

Assume that $F_2$ occurs. The analysis is the same as the case where $F_    1$ occurs.

Assume that $F_3$ occurs. Note that $\mathbb{P}(D_3 \cap F_3)=(1-p)^{4m}$.\qedhere 

\vspace{2mm}

\end{proof}

Now let us move on to estimating the term $\mathbb{P}(E^c \cap A_2 \cap B)$. 
\begin{lemma} \label{lemma: E A2 B}
The probability that $[2m]^3$ is bad along with the occurrence of the event that there are exactly two subcubes are bad is correlated with double empty line segments of length $2m$. More precisely, we have 

$$\mathbb{P}(E^c \cap A_2) \leq C(d)m^2(1-p)^{4m-8}.$$
\end{lemma}
\begin{proof}
In order to estimate $\mathbb{P}(E^c \cap A_2 \cap B)$ we will separate cases based on the locations of the two bad subcubes. 

\vspace{2mm}
Case 1: Two bad cubes are adjacent. 

Without loss of generality assume the subcubes $C_1$ and $C_2$ are bad, as in Figure \ref{figure6}.  

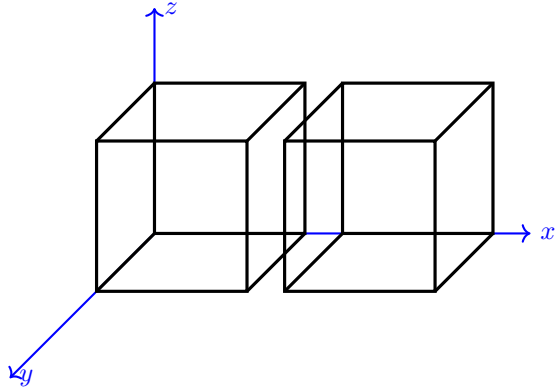
\begin{figure}[htbp]
\centering
\begin{tikzpicture}
		[cube/.style={very thick,black},
			grid/.style={very thin,gray},
			axis/.style={->,blue,thick}]

			
	\draw[axis] (0,0,0) -- (5,0,0) node[anchor=west]{$x$};
	\draw[axis] (0,0,0) -- (0,3,0) node[anchor=west]{$z$};
	\draw[axis] (0,0,0) -- (0,0,5) node[anchor=west]{$y$};

	\draw[cube] (0,0,0) -- (0,2,0) -- (2,2,0) -- (2,0,0) -- cycle;
	\draw[cube] (0,0,2) -- (0,2,2) -- (2,2,2) -- (2,0,2) -- cycle;

        \draw[cube] (0,0,0) -- (0,0,2);
	\draw[cube] (0,2,0) -- (0,2,2);
	\draw[cube] (2,0,0) -- (2,0,2);
	\draw[cube] (2,2,0) -- (2,2,2);

        \draw[cube] (2.5,0,0) -- (2.5,2,0) -- (4.5,2,0) -- (4.5,0,0) -- cycle;
	\draw[cube] (2.5,0,2) -- (2.5,2,2) -- (4.5,2,2) -- (4.5,0,2) -- cycle;
        \draw[cube] (2.5,0,0) -- (2.5,0,2);
	\draw[cube] (2.5,2,0) -- (2.5,2,2);
	\draw[cube] (4.5,0,0) -- (4.5,0,2);
	\draw[cube] (4.5,2,0) -- (4.5,2,2);
 
\end{tikzpicture}
\caption{Bad subcubes $C_1$ and $C_2$.} 
\label{figure6}
\end{figure}

The main idea of the proof is somewhat similar to that of Lemma ~\ref{lemma: E A1 B}. We will construct 5 sets of events $T_{start},T_{x=1},T_{y=1}$, $T_{z=1}$ and $T_{x=2m}$ such that, if the events in these 5 sets occur along with events $A_2$ and $B$, then the event $E$ also occurs; that is, the cube $[2m]^d$ is good. Moreover, the probability that some of these events in these 5 sets fails to occur is approximately $(1-p)^{4m}$. If the probability that some of the events in these 5 sets fails to occur is significantly larger than $(1-p)^{4m}$, then a more detailed analysis is required. The rough idea is that other events, not contained in $T_{\text{start}}$, $T_{x=1}$, $T_{y=1}$, $T_{z=1}$,or $T_{x=2m}$ can also lead to percolation. The probability of the complement of these additional events, together with the complement of the events in $T_{\text{start}}$, $T_{x=1}$, $T_{y=1}$,$T_{z=1}$,and $T_{x=2m}$ is  $\leq (1-p)^{4m}$. The exact details are presented below.

\vspace{2mm}
 The infection spreads as follows. First, after at most $m^3$ steps, the interior of $C_3,\dots,C_8$ have been infected since $C_3,\dots, C_8$ are good. Then the sides of the subcubes $C_3, \dots, C_8$ become infected (except for those on $[(1,2m),(1),(1,2m)]$, $[(1),(1,2m),(1,2m)]$, $[(1,2m),(1,2m),(1)]$ and $[(2m),(1,2m),(1,2m)]$ and those belonging to the sides of $[2m]^d$ )after at most $m^3+Cm$ steps, due to the occurrence of the events in $S_{start}$. Here it comes to the difference between this case and the case where there is only one subcube is bad. We will need to define a new event $H'$ to be \{every two adjacent line segments of the form $[(2,2m-1),(y),(z)]$ of length $2m-2$ in $[(1,2m),(1,m),(1,m)]$ parallel to the $x$-axis is non-empty\}, where $y, z \in \{1,2,\dots,m\}$.
 
 Then, after at most $3m^3+Cm$ steps the interior of $C_1$ and $C_2$ have been infected due to the occurrence of the events $H'$. Finally, the uninfected vertices on $[(1,2m),(1),(1,2m)]$, $[(1),(1,2m),(1,2m)]$ and $[(1,2m),(1,2m),(1)]$ will be infected after at most $3m^3+Cm'$ steps, due to the occurrence of the event in $T_{x=1},T_{y=1}$, $T_{z=1}$ and $T_{x=2m}$.

It is easy to see that $\mathbb{P}(H'^c) \leq C m^2(1-p)^{4m-4}$. 

\vspace{2mm}
Now we will start defining the events in the set $T_{start}$. It turns out that $B_8$ and $B_{10}$ in the set $T_{start}$ serve our purpose. It is easy to see that after at most $2m^3+Cm$ steps, every vertex in $[(1,2m),(1,2m),(1,2m)]$ will be infected except for the vertices in $[(1),(1,2m),(1,2m)]$,$[(1,2m),(1),(1,2m)]$, $[(1,2m),(1,2m),(1)]$, and $[(2m),(1,2m),(1,2m)]$\} and the sides of $[2m]^3$ due to the occurrence of $A_2,B_8,B_{10}$ and $H'$. 

We can apply the same method from Lemma \ref{eq: P(E A1 d)} to infect uninfected vertices in $[(1),(1,2m),(1,2m)]$ , $[(1,2m),(1,2m),(1)]$, and $[(2m),(1,2m),(1,2m)]$\}. Therefore, we just need to focus on infecting the vertices on $[(1,2m),(1),(1,2m)]$.

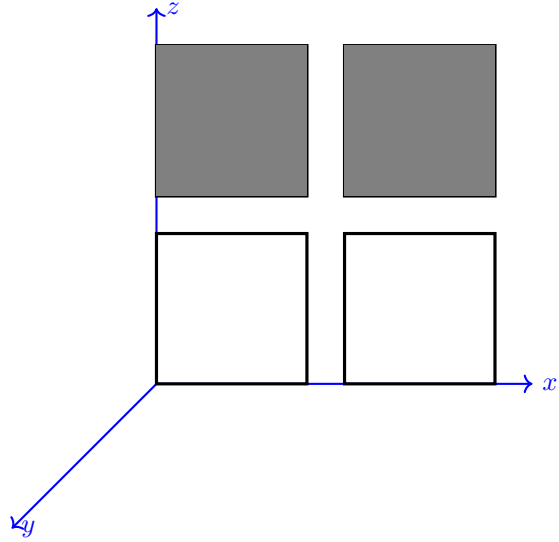
\begin{figure}[htbp]
\centering
\begin{tikzpicture}
		[cube/.style={very thick,black},
			grid/.style={very thin,gray},
			axis/.style={->,blue,thick}]

			
	\draw[axis] (0,0,0) -- (5,0,0) node[anchor=west]{$x$};
	\draw[axis] (0,0,0) -- (0,5,0) node[anchor=west]{$z$};
	\draw[axis] (0,0,0) -- (0,0,5) node[anchor=west]{$y$};

	\draw[cube] (0,0,0) -- (0,2,0) -- (2,2,0) -- (2,0,0) -- cycle;
	\draw[cube] (0,2.5,0) -- (0,4.5,0) -- (2,4.5,0) -- (2,2.5,0) -- cycle;
       \draw[cube] (2.5,0,0) -- (2.5,2,0) -- (4.5,2,0) -- (4.5,0,0) -- cycle;
       \draw[cube] (2.5,2.5,0) -- (2.5,4.5,0) -- (4.5,4.5,0) -- (4.5,2.5,0) -- cycle;

  \fill[gray] (0,2.5,0) -- (0,4.5,0) -- (2,4.5,0) -- (2,2.5,0) -- cycle;
  \fill[gray] (2.5,2.5,0) -- (2.5,4.5,0) -- (4.5,4.5,0) -- (4.5,2.5,0) -- cycle;
  
\end{tikzpicture}
\caption{Configuration on $[(1,2m),(1),(1,2m)]$ after $3m^3+Cm^2$ steps.} 
\label{figure9}

\end{figure}
Consider the vertices in the region $[(1,2m),(1),(1,2m)]$ as shown in Figure~\ref{figure9}, where the vertices in the shaded area are infected, and those in the white area remain uninfected. Note that since the subcubes $C_5$ and $C_6$ are good, the only uninfected vertices are in the sides of $C_5$ and $C_6$ and $[(1,2m)(1)(1,m)]$ after $3m^3+Cm^2$ steps.

\vspace{2mm}
Let us define the events in $T_{y=1}$. Let 

$E_1(y=1):=$\{every two adjacent line segments of length $m-2$ in $[(2,m-1),(1),(1,2m)]$ parallel to the $x$-axis is non-empty\},

\vspace{1mm}
$E_2(y=1):=$\{one of $[(m),(1),(2,2m-1)]$ or $[(m+1),(1),(2,2m-1)]$ is non-empty\},

\vspace{1mm}

and $E_3(y=1):=$\{every two adjacent line segments of length $m-2$ in $[(m+2,2m-1),(1),(1,2m)]$ parallel to the $x$-axis is non-empty\}.

\vspace{2mm}
It is clear that if the event $E_2(y=1)$ occurs along with either the event $E_1(y=1)$ or the event $E_3(y=1)$ then every vertex in  $[(1,2m),(1),(1,2m)]$ will be infected except for the sides of $[2m]^3$. 

Note that $\mathbb{P}(E_1(y=1)^c \cap E_3(y=1)^c)= (1-p)^{4m-8}$ and $\mathbb{P}(E_2(y=1)^c)= (1-p)^{4m-4}$. By Fact \ref{fact}, we have the desired result. 

\vspace{2mm}
Case 2: Two bad subcubes are non-adjacent and not located diagonally. 

Without loss of generality, assume that the subcubes $C_1$ and $C_4$ are bad as in Figure \ref{figure7}. 

The main idea is as follows. We will define two new events $E_4$ and $E_5$ such that along with the occurrence of the events $A_2$ and $B$, every vertex in $[(1,2m),(1,2m),(1,2m)]$ has been infected except for those belonging to $[(1),(1,2m),(1,2m)]$, $[(1,2m),(1),(1,2m)]$, $[(1,2m),(1,2m),(1)]$, $[(2m),(1,2m),(1,2m)]$, $[(1,2m),(2m),(1,2m)]$ and the edges of $[2m]^3$ after at most $2m^3+Cm$ steps. Another requirement for the events $E_4$ and $E_5$ is that $\mathbb{P}(E_4^c \cap E_5^c)$ is roughly  $(1-p)^{4m}$. Then consider the uninfected vertices on  $[(1),(1,2m),(1,2m)]$, $[(1,2m),(1),(1,2m)]$, $[(1,2m),(1,2m),(1)]$, $[(2m),(1,2m),(1,2m)]$, $[(1,2m),(2m),(1,2m)]$ and the edges of $[2m]^3$ after $2m^3+Cm$ steps. The analysis of infection for these uninfected vertices is the same as in the cases where the two bad subcubes are adjacent or where there is only one bad subcube.

Now let us define events $E_4$ and $E_5$. 

$E_4:=$ \{every two adjacent line segments of length $m-2$ on $[(m+2,2m-1),(2,m+1),(m)]$ parallel to the $x$-axis is non-empty\},  
and 

$E_5:=$\{ at least one of the two line segments $[(m),(2,2m-1),(m)]$ or $[(m+1),(2,2m-1),(m)]$ is non-empty\}. 


\begin{figure}[htbp]
 \centering

\begin{tikzpicture}
		[cube/.style={very thick,black},
			grid/.style={very thin,gray},
			axis/.style={->,blue,thick}]

			
	\draw[axis] (0,0,0) -- (5,0,0) node[anchor=west]{$x$};
	\draw[axis] (0,0,0) -- (0,3,0) node[anchor=west]{$z$};
	\draw[axis] (0,0,0) -- (0,0,5) node[anchor=west]{$y$};

	\draw[cube] (0,0,0) -- (0,2,0) -- (2,2,0) -- (2,0,0) -- cycle;
	\draw[cube] (0,0,2) -- (0,2,2) -- (2,2,2) -- (2,0,2) -- cycle;



        \draw[cube] (2.5,0,2.5) -- (2.5,2,2.5) -- (4.5,2,2.5) -- (4.5,0,2.5) -- cycle;
	\draw[cube] (2.5,0,4.5) -- (2.5,2,4.5) -- (4.5,2,4.5) -- (4.5,0,4.5) -- cycle;
 
	\draw[cube] (0,0,0) -- (0,0,2);
	\draw[cube] (0,2,0) -- (0,2,2);
	\draw[cube] (2,0,0) -- (2,0,2);
	\draw[cube] (2,2,0) -- (2,2,2);



\draw[cube] (2.5,0,2.5) -- (2.5,0,4.5);
	\draw[cube] (2.5,2,2.5) -- (2.5,2,4.5);
	\draw[cube] (4.5,0,2.5) -- (4.5,0,4.5);
	\draw[cube] (4.5,2,2.5) -- (4.5,2,4.5);
	
\end{tikzpicture}
\caption{Bad subcubes $C_1$ and $C_4$.} 
\label{figure7}
\end{figure}
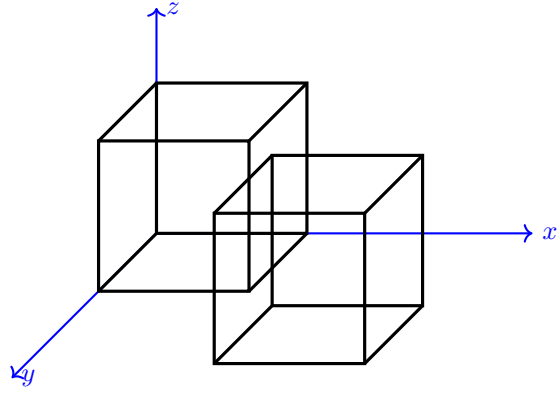

Case 3: Two bad subcubes are located diagonally. 

Without loss of generality we can assume the subcubes $C_3$ and $C_6$ are bad as in Figure \ref{figure8}.
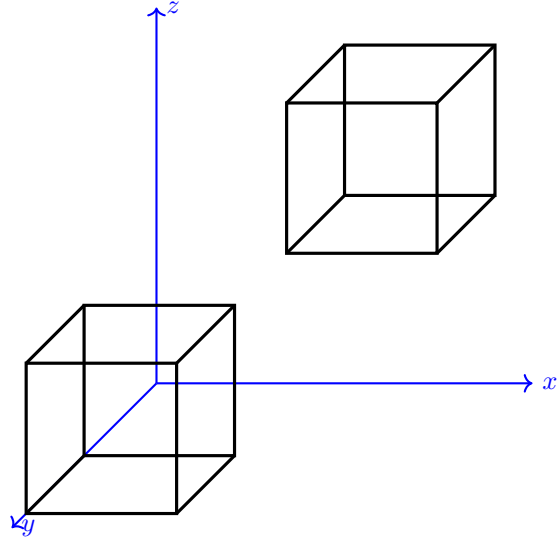
\begin{figure}[htbp]
\centering
\begin{tikzpicture}
		[cube/.style={very thick,black},
			grid/.style={very thin,gray},
			axis/.style={->,blue,thick}]

			
	\draw[axis] (0,0,0) -- (5,0,0) node[anchor=west]{$x$};
	\draw[axis] (0,0,0) -- (0,5,0) node[anchor=west]{$z$};
	\draw[axis] (0,0,0) -- (0,0,5) node[anchor=west]{$y$};



       \draw[cube] (0,0,2.5) -- (0,2,2.5) -- (2,2,2.5) -- (2,0,2.5) -- cycle;
	\draw[cube] (0,0,4.5) -- (0,2,4.5) -- (2,2,4.5) -- (2,0,4.5) -- cycle;




        \draw[cube] (0,0,2.5) -- (0,0,4.5);
	\draw[cube] (0,2,2.5) -- (0,2,4.5);
	\draw[cube] (2,0,2.5) -- (2,0,4.5);
	\draw[cube] (2,2,2.5) -- (2,2,4.5);


 \draw[cube] (2.5,2.5,0) -- (2.5,4.5,0) -- (4.5,4.5,0) -- (4.5,2.5,0) -- cycle;
    \draw[cube] (2.5,2.5,2) -- (2.5,4.5,2) -- (4.5,4.5,2) -- (4.5,2.5,2) -- cycle;
        \draw[cube] (2.5,2.5,0) -- (2.5,2.5,2);
	\draw[cube] (2.5,4.5,0) -- (2.5,4.5,2);
	\draw[cube] (4.5,2.5,0) -- (4.5,2.5,2);
	\draw[cube] (4.5,4.5,0) -- (4.5,4.5,2);
	
\end{tikzpicture}
\caption{Bad subcubes $C_3$ and $C_6$.} 
\label{figure8}
\end{figure}

The main idea is as follows. We will define three new events $E_6$ , $E_7$ and $E_8$ such that along with the occurrence of the events $A_2$ and $B$, every vertex in $[(1,2m),(1,2m),(1,2m)]$ has been infected except for those belonging to  $[(1),(1,2m),(1,2m)]$, $[(1,2m),(1),(1,2m)]$, $[(1,2m),(1,2m),(1)]$, $[(2m),(1,2m),(1,2m)]$, $[(1,2m),(2m),(1,2m)]$, $[(1,2m),(1,2m),(2m)]$ and the sides of $[2m]^3$ after at most $2m^3+Cm$ steps. Another requirement for the events $E_6$, $E_7$ $E_8$ is that $\mathbb{P}(E_i^c)$ is roughly  $(1-p)^{4m}$ for $i \in \{6,7,8\}$. Then consider the uninfected vertices on  $[(1),(1,2m),(1,2m)]$, $[(1,2m),(1),(1,2m)]$, $[(1,2m),(1,2m),(1)]$, $[(2m),(1,2m),(1,2m)]$, $[(1,2m),(2m),(1,2m)]$ except for these on the sides of $[2m]^3$ after $2m^3+Cm$ steps. The analysis of infection for these uninfected vertices is the same as in the case where there is only one bad subcube.

Let us define the events $E_6$,$E_7$, and $E_8$. 

$E_6$:= \{at least one of the 4 line segments of 
$[(m),(1,m),(m)],[(m),(1,m),(m+1)],[(m+1),(1,m),(m)],$ or $[(m+1),(1,m),(m+1)]$ is non-empty\}, 

$E_7$:=  \{at least one of the 4 line segments $[(1,m),(m),(m)] [(1,m),(m),(m+1)],
[(1,m),(m+1),(m)],$ or  $[(1,m),(m+1),(m+1)]$ is non-empty\},

$E_8$=\{at least one of the 4 line segments $[(m+1,2m),(m),(m)], [(m+1,2m),(m),(m+1)],
[(m+1,2m),(m+1),(m)],$ or $[(m+1,2m),(m+1),(m+1)] $ is non-empty\}. \qedhere

\end{proof}

Now we are in a position to prove Lemma \ref{lemmaV1}, which is the main lemma in this subsection. 
\begin{proof}
In order to estimate the probability of the event $E^c$ we will write the event $E^c$ in a different way and so we have, by using Lemma \ref{lemma: A B}, Lemma \ref{lemma: E A1 B} and Lemma \ref{lemma: E A2 B}, 
\begin{align*}
  \mathbb{P}(E^c) &= \mathbb{P}(E^c \cap A \cap B)+\mathbb{P}(E^c \cap A^c \cap B)+\mathbb{P}(E^c \cap A \cap B^c)+\mathbb{P}(E^c \cap A^c \cap B^c) \\
  &\leq 2 \mathbb{P}(B^c)+\mathbb{P}(E^c \cap A^c\cap B)\\
  & = 2 \mathbb{P}(B^c) + \mathbb{P}(\cup_{i=1}^8(E^c \cap A_i\cap B)) \\
  & \leq C' (1-p)^{4m}+\mathbb{P}(E^c \cap A_1 \cap B) + \mathbb{P}(E^c \cap A_2 \cap B) + \sum_{i=3}^8 \mathbb{P}(A_i) \\
  & \leq C' (1-p)^{4m}+\mathbb{P}(E^c \cap A_1 \cap B) + \mathbb{P}(E^c \cap A_2 \cap B) + \sum_{i=3}^8 \binom{8}{i} \eta_{m,3}^i \\
  & \leq C' (1-p)^{4m}+\mathbb{P}(E^c \cap A_1 \cap B) + \mathbb{P}(E^c \cap A_2 \cap B)+ C \eta_{m,3}^3\\
  & \leq Bm^2(1-p)^{4m-8}+ C \eta_{m,3}^3
\end{align*}
where $C$ and $C'$ are absolute constants. \qedhere
\end{proof}

\subsection{$d=r \geq 3$}
 This lemma gives a recursive relation on the probability that $[m]^d$ is bad. As mentioned in the introduction, we prove Lemma ~\ref{lemmaV3} using induction, where Lemma ~\ref{lemmaV1} is the base case.

\begin{lemma} \label{lemmaV3} 
Let $\eta_{m,r}$ be the probability that $[m]^d$ is bad with the infection threshold $r=d$. We have
   $$\eta_{2m,d} \leq C(d)\eta_{m,d}^3+C(d)m^{d-1}(1-p)^{4m-8},$$
\end{lemma}

The intuition behind the proof of Lemma ~\ref{lemmaV3} is as follows. We begin by partitioning the cube $[2m]^{d+1}$ into $2^{d+1}$ disjoint subcubes and break the proof into several subcases according to the number of bad subcubes. Specifically, we analyze the cases where exactly one or two of the $2^{d+1}$ subcubes are bad as well as the case where none are bad.  The main technical part of the proof is devoted to the analysis to these three subcases, where we study the interaction arising among the $2^{d+1}$ subcubes by using the induction on the dimension $d$.   Further the probability of having at least three bad subcubes can be estimated by a very rough upper bound of $C\eta_{m,d+1}^3$, as established in Lemma ~\ref{lemmaV3}.

 Before proving this lemma we need to introduce some notation and prove several auxiliary lemmas. 

Consider partitioning $[2m]^d$ into $2^d$ subcubes, i.e
$$[(a_1,a_1+m-1),(a_2,a_2+m-1),...,(a_d,a_d+m-1)]$$ where $a_i=\{1,m+1 \}$ for $i \in [d]$.

First we need to define a few events.
\begin{itemize}
    \item $E^d:=[2m]^d$ is good.

\item $A^{d}:$=\{all $2^d$  subcubes are good \}.

 \item $A_i^{d}:$=\{exactly $i$ \text{subcubes among $2^d$ subcubes are bad} \} for $i \in [2^d]$. 
\end{itemize}
We would like to prove that for $d \geq 3$
\begin{align}
\mathbb{P}((E^d)^c \cap A^d) \leq C(d)(1-p)^{4m},   \label{eq: P(E A d)}
\end{align}
\begin{align}
  \mathbb{P}((E^d)^c \cap A_{1}^d) \leq C(d)m^{d-3}(1-p)^{4m},  \label{eq: P(E A1 d)}
\end{align}
and 
\begin{align}
    \mathbb{P}((E^d)^c \cap A_2^d) \leq C(d)m^{d-1}(1-p)^{4m-8}. \label{eq: P(E A2 d)}
\end{align}

Then we have
\begin{align*}
  \mathbb{P}((E^d)^c) &= \mathbb{P}((E^d)^c \cap A^d)+\mathbb{P}((E^d)^c \cap (A^d)^c )\\
  &= \mathbb{P}((E^d)^c \cap A^d)+\mathbb{P}(\cup_{i=1}^{2^d}((E^d)^c \cap A_i^d))\\
  & =\mathbb{P}((E^d)^c \cap A^d)+\mathbb{P}((E^d)^c \cap A_1^d) + \mathbb{P}((E^d)^c \cap A_2^d) + \sum_{i=3}^{2^d} \mathbb{P}(A_i^d) \\
  & \leq C(d)m^{d-1}(1-p)^{4m-8}+ C(d) \eta_{m,3}^3
\end{align*} 

When $d=3$ it is easy to see that (\ref{eq: P(E A d)})(\ref{eq: P(E A1 d)})(\ref{eq: P(E A2 d)}) are satisfied by Lemmas (\ref{lemmaV1})(\ref{lemma: E A1 B})(\ref{lemma: E A2 B}). Now assume that (\ref{eq: P(E A d)})(\ref{eq: P(E A1 d)})(\ref{eq: P(E A2 d)}) are satisfied for $d \geq 3$ we would like to show they are satisfied for $d+1$. 

Before proceeding, we need some notation. Let a $d$-dimensional subgrid $ [ \underbrace{(1,2m),(1,2m),\dots, (1,2m) }_{d},(i)]$ be denoted by $B((1,2m),d,i)$ and a $d$-dimensional subgrid $ [ \underbrace{(1,m),(1,m),\dots, (1,m) }_{d},(i)] $ by $B((1,m),d,i)$.

Let us define $E_i^{d}$ to be the event that $(B(1,2m),d,i)$ is good.

Let $B\left(a_1,a_2,\cdots, a_{i-1}, a_{i+1},\cdots,a_{j-1},a_{j+1},\cdots,a_d,(1,m)\right)$ be the event that at least one of the 4 line segments 
$$[(a_1),(a_2)\cdots,(a_{i-1}),(m),(a_{i+1}),\cdots,(a_{j-1}),(m),(a_{j+1}),\cdots,(a_d),(1,m)]$$
$$[(a_1),(a_2)\cdots,(a_{i-1}),(m),(a_{i+1}),\cdots,(a_{j-1}),(m+1),(a_{j+1}),\cdots,(a_d),(1,m)]$$
$$[(a_1),(a_2)\cdots,(a_{i-1}),(m+1),(a_{i+1}),\cdots,(a_{j-1}),(m),(a_{j+1}),\cdots,(a_d),(1,m)]$$
$$[(a_1),(a_2)\cdots,(a_{i-1}),(m+1),(a_{i+1}),\cdots,(a_{j-1}),(m+1),(a_{j+1}),\cdots,(a_d),(1,m)]$$
is non-empty, where $a_l \in \{1,m,m+1,2m\}$ for $l \in [d] \backslash \{i,j\}$. 

The event $B\left(a_1,a_2,\cdots, a_{i-1}, a_{i+1},\cdots,a_{j-1},a_{j+1},\cdots,a_d,(m+1,2m)\right)$ is defined similarly. 

Let $B(\sigma\left(a_1,a_2,\cdots, a_{i-1}, a_{i+1},\cdots,a_{j-1},a_{j+1},\cdots,a_d,(1,m)\right))$ denote the event that at least one of the 4 line segments
$$\sigma[(a_1),(a_2)\cdots,(a_{i-1}),(m),(a_{i+1}),\cdots,(a_{j-1}),(m),(a_{j+1}),\cdots,(a_d),(1,m)]$$
$$\sigma[(a_1),(a_2)\cdots,(a_{i-1}),(m),(a_{i+1}),\cdots,(a_{j-1}),(m+1),(a_{j+1}),\cdots,(a_d),(1,m)]$$
$$\sigma[(a_1),(a_2)\cdots,(a_{i-1}),(m+1),(a_{i+1}),\cdots,(a_{j-1}),(m),(a_{j+1}),\cdots,(a_d),(1,m)]$$
$$\sigma[(a_1),(a_2)\cdots,(a_{i-1}),(m+1),(a_{i+1}),\cdots,(a_{j-1}),(m+1),(a_{j+1}),\cdots,(a_d),(1,m)]$$
is nonempty, where $a_l \in \{1,m,m+1,2m\}$ for $l \in [d] \backslash \{i,j\}$ and $\sigma \in \mathfrak{S}_{d+1}$.

The event $B(\sigma\left(a_1,a_2,\cdots, a_{i-1}, a_{i+1},\cdots,a_{j-1},a_{j+1},\cdots,a_d,(m+1,2m)\right))$ is defined similarly, where $\sigma \in \mathfrak{S}_{d+1}$.

Let us define $D_{\sigma}(1,m)$ and $D_{\sigma}(1,2m)$ to be

$$D_{\sigma}(1,m):= B(\sigma\left(a_1,a_2,\cdots, a_{i-1}, a_{i+1},\cdots,a_{j-1},a_{j+1},\cdots,a_d,(1,m)\right)),$$
and
$$D_{\sigma}(m+1,2m):=B(\sigma\left(a_1,a_2,\cdots, a_{i-1}, a_{i+1},\cdots,a_{j-1},a_{j+1},\cdots,a_d,(m+1,2m)\right)).$$

Let us define $D$ to be 
\begin{align*}
D := \cap_{l \in [d+1] \backslash \{i,j\}}\cap_{a_l \in \{1,m,m+1,2m\}} \cap_{\sigma \in \mathfrak{S}_{d+1}} \big(D_{\sigma}(1,m) \cap D_{\sigma}(m+1,2m) \big). 
\end{align*}




Let $B'$ be defined as 
$$B':=\cap_{l \in [d+1] \backslash \{i,j\}}\cap_{a_l \in \{1,m,m+1,2m\}} \cap_{\sigma \in \mathfrak{S}_{d+1}} D_{\sigma}(m+1,2m).$$


Let $D\left(a_1,a_2,\cdots, a_{i-1}, a_{i+1},\cdots,a_d, m \right)$ be the event that at least one of the 4 line segments 
$$[(a_1),(a_2)\cdots,(a_{i-1}),(1,m),(a_{i+1}),\cdots,(a_d),(m)]$$
$$[(a_1),(a_2)\cdots,(a_{i-1}),(m+1,2m),(a_{i+1}),\cdots,(a_d),(m)]$$
$$[(a_1),(a_2)\cdots,(a_{i-1}),(1,m),(a_{i+1}),\cdots,(a_d),(m+1)]$$
$$[(a_1),(a_2)\cdots,(a_{i-1}),(m+1,2m),(a_{i+1}),\cdots,(a_d),(m+1)]$$
is non-empty, where $a_i \in \{1,2m\}$ for $i \in [d]$. 

The event $D(\sigma (a_1,a_2,\cdots, a_{i-1}, a_{i+1},\cdots,a_d, m))$ is defined similarly.

Let us define $D'$ to be 
$$D':=\cap_{j \in [d+1] \backslash {i}]}\cap_{a_j \in \{1,2m\}} \cap_{\sigma \in \mathfrak{S}_{d+1}} D\left(a_1,a_2,\cdots, a_{i-1}, a_{i+1},\cdots,a_d, m \right)$$

Let $E\left(a_1,a_2,\cdots, a_{i-1}, a_{i+1},\cdots,a_d,m\right)$ be an event that one of the 4 segments 
$$[(a_1),(a_2)\cdots,(a_{i-1}),(m),(a_{i+1}),\cdots,(a_d),(1,m)]$$
$$[(a_1),(a_2)\cdots,(a_{i-1}),(m+1),(a_{i+1}),\cdots,(a_d),(1,m)]$$
$$[(a_1),(a_2)\cdots,(a_{i-1}),(m),(a_{i+1}),\cdots,(a_d),(m+1,2m)]$$
$$[(a_1),(a_2)\cdots,(a_{i-1}),(m+1),(a_{i+1}),\cdots,(a_d),(m+1,2m)]$$

is non-empty, where $a_i \in \{1,2m\}$ for $i \in [d]$. 

The event $E\left(\sigma (a_1,a_2,\cdots, a_{i-1}, a_{i+1},\cdots,a_d,m)\right)$ is defined similarly. 

Let $E(1,2m)$ to be defined as 
$$E(1,2m):=\cap_{l \in [d+1] \backslash \{i\}}\cap_{a_l \in \{1,2m\}} \cap_{\sigma \in \mathfrak{S}_{d+1}}E\left(\sigma (a_1,a_2,\cdots, a_{i-1}, a_{i+1},\cdots,a_d,m)\right).$$

 Let us first deal with the case where all $2^{d+1}$ subcubes are good, where the subcubes are $[(a_1,a_1+m-1),(a_2,a_2+m-1),...,(a_{d+1},a_{d+1}+m-1)]$ where $a_i=\{1,m+1 \}$ for $i \in [d+1]$.

\begin{lemma} \label{lemma: E^d A^d}
The probability that $[2m]^d$ is bad along with the occurrence of the event that all subcubes are good is correlated with double empty line segments of length $2m$. More precisely, we have 
    \begin{align*}
\mathbb{P}((E^d)^c \cap A^d) \leq C(d)(1-p)^{4m}. 
\end{align*}
\end{lemma}
\begin{proof}

The idea for the proof is as follows. We will construct some events in the set $B_{start}$ such that with the occurrence of these events and the event $A^{d+1}$ then the event $E^{d+1}$ occur. Additionally, the probability of the complement of these events is roughly $(1-p)^{4m}$.

\vspace{1mm}

Assume that all $2^{d+1}$ subcubes are good, i.e $$[(a_1,a_1+m-1),(a_2,a_2+m-1),...,(a_{d+1},a_{d+1}+m-1)]$$ where $a_i=\{1,m+1 \}$ for $i \in [d+1]$,  are good. 

After $m^{d+1}$ steps, the only possibly uninfected vertices on $B((1,2m),d,i)$ for $i \notin \{1,m,m+1,2m\}$ are $[(a_1),(a_2),\cdots,(a_{d}),(i)]$, where $a_j \in \{1,m,m+1,2m\}$ for $j \in [d]$.

Now consider the possibly uninfected vertices on $B((1,2m),d,i)$ for $i \in \{1,m,m+1,2m\}$ after $m^{d+1}$ steps. Note that the vertices on the sides of subgrids 
$$[(a_1,a_1+m-1),(a_2,a_2+m-1),...,(a_{d},a_{d}+m-1),(j)],$$
where $a_i \in \{1,m+1\}$ for $i \in [d]$ and $j \in \{1,m,m+1,2m\}$, may have not been infected. 

Observe that the vertices in the sides of the subcubes $B((1,2m),d,i)$ where $i \in \{m,m+1\}$ do not belong to the sides of $[2m]^{d+1}$ except for the vertices of the form 
$$[(a_1),(a_2),\cdots, (a_d),(i)],$$
where $a_j \in \{1,m,m+1,2m\}$ for $j \in [d]$ and $i \in \{m,m+1\}$.

It turns out that  the events $D$, $D'$ and $E(1,2m)$ in $B_{start}$ serve our purpose.

\vspace{1mm}
With the occurrence of $D$ and $A^{d+1}$, after at most $m^{d+1}+C(d)m$ steps, every vertex on the subcube $B((1,2m),d,i)$ with $i \notin \{1,m,m+1,2m\}$ has been infected except for those belonging to the sides of $[2m]^{d+1}$ and the vertices in $[(a_1),(a_2),\cdots,(a_{i-1}),(m),(a_{i+1}),\cdots, (a_d),(1,2m)]$ and $[(a_1),(a_2),\cdots,(a_{i-1}),(m+1),(a_{i+1}),\cdots, (a_d),(1,2m)]$ where $a_j \in \{1,2m\}$ for $j \in [d]$.

\vspace{1mm}
Now every vertex, except possibly for the vertices on the sides of $B((1,2m),d,i)$, on subcubes $B((1,2m),d,i)$ for $i \in \{1,m,m+1,2m\}$ has exactly one infected neighbor which does not belong to $B((1,2m),d,i)$ for $i \in \{1,m,m+1,2m\}$. Hence we can use the induction hypothesis and thus we have
\begin{align*}
    \mathbb{P}((E^{d+1})^c \cap A^{d+1}) & \leq \mathbb{P}((E^{d+1})^c \cap A^{d+1} \cap D \cap E_{1}^d \cap E_{m}^d \cap E_{m+1}^d \cap E_{2m}^d)\\
    & + \mathbb{P}((E^{d+1})^c \cap A^{d+1} (\cap D \cap E_{1}^d \cap E_{m}^d \cap E_{m+1}^d \cap E_{2m}^d)^c)\\
    & \leq \mathbb{P}((E^{d+1})^c \cap A^{d+1} \cap D \cap E_{1}^d \cap E_{m}^d \cap E_{m+1}^d \cap E_{2m}^d) + 4 \mathbb{P}((E_1^d)^c)+ \mathbb{P}(D^c)\\
    &\leq \mathbb{P}((E^{d+1})^c \cap A^{d+1} \cap D \cap E_{1}^d \cap E_{m}^d \cap E_{m+1}^d \cap E_{2m}^d) + C(1-p)^{4m}
\end{align*}

Now we need to handle the term $\mathbb{P}((E^{d+1})^c \cap D \cap A^{d+1} \cap E_{1}^d \cap E_{m}^d \cap E_{m+1}^d \cap E_{2m}^d)$.


\vspace{1mm}

It clear that with the occurrence of the events $D, A^{d+1}, E_{1}^d, E_{m}^d, E_{m+1}^d, E_{2m}^d$ and $D'$, every vertex on $[2m]^{d+1}$ has been infected after at most $m^{d+1}+C(d)m$ steps, except for those on the sides of $[2m]^{d+1}$.


Thus we have 
$$\mathbb{P}((E^{d+1})^c \cap D \cap A^{d+1} \cap E_{1}^d \cap E_{m}^d \cap E_{m+1}^d \cap E_{2m}^d \cap D' \cap E(1,2m))=0$$

Note that 
$$\mathbb{P}((D')^c) \leq C(d)(1-p)^{4m},$$
and 
$$\mathbb{P}(E(1,2m)^c) \leq C(d)(1-p)^{4m}.$$

By Fact \ref{fact}, we have the desired result. \qedhere

\end{proof}
Now we will move onto the case where there is exactly one bad subcube and we will show the following lemma. 

\begin{lemma} \label{lemma: E^d A_1^d}
The probability that $[2m]^d$ is bad along with the occurrence of the event that there is exactly one bad subcube is correlated with double empty line segments of length $2m$. More precisely, we have 
    \begin{align*}   
  \mathbb{P}((E^d)^c \cap A_{1}^d) \leq C(d)m^{d-3}(1-p)^{4m}.
\end{align*}
\end{lemma}

\begin{proof}
There is exactly one bad subcube, i.e 
$$[(a_1,a_1+m-1),(a_2,a_2+m-1),...,(a_{d+1},a_{d+1}+m-1)]$$ where $a_i=\{1,m+1 \}$ for $i \in [d+1]$ are good except for one.

W.l.o.g, we can assume that the subcube $[m]^{d+1}$ is bad. 

The main idea of the proof is somewhat similar to that in the proof of Lemma \ref{lemma: E A1 B}. We will construct an event such that every vertex in $B((1,2m), d,i)$ has been infected for  all $i \in \{m+2,\cdots,2m-1\}$ after at most $m^{d+1}+C(d)m$ steps except for those vertices on the sides of $B((1,2m),d,i)$ for $i \in \{m+2,\cdots,2m-1\}$. For the uninfected vertices on $B((1,2m),d,m+1)$ and $B((1,2m),d,2m)$ we can use induction hypothesis because every vertex in $B((1,2m),d,m+1)$ and $B((1,2m),d,2m)$ has at least one infected neighbor which does not belong to $B((1,2m),d,m+1)$ and $B((1,2m),d,2m)$. The same can be applied to infecting vertices on $B((1,2m), d,i)$ for $i \in [m]$. Then we will deal with the rest of the uninfected vertices that are not in the sides of $[2m]^{d+1}$

\vspace{1mm}

The exact details are presented below.

After $m^{d+1}$ steps let us describe the uninfected vertices on $[2m]^{d+1}$. Every vertex on $B((1,2m),d,i)$ for $i \in \{m+2,m+3,\cdots,2m-1\}$ has been infected except for those of the form $[(a_1),(a_2),\cdots,(a_{d}),(i)]$, where $a_j \in \{1,m,m+1,2m\}$ for $j \in [d]$.

\vspace{1mm}
For the vertices on $B((1,2m),d,j)$ where $j \in \{m+1,2m\}$, all of them have been infected except for those on the sides of $$[(a_1,a_1+m-1),(a_2,a_2+m-1),...,(a_d,a_d+m-1),(j)]$$ where $a_i=\{1,m+1 \}$ for $i \in [d]$. 

\vspace{1mm}
For the vertices on $B((1,2m),d,j)$ where $j \in \{1,m\}$, the vertices on $B((1,m),d,1)$and $B((1,m),d,m)$ have not been infected. Additionally, the vertices on the edges of $$[(a_1,a_1+m-1),(a_2,a_2+m-1),...,(a_d,a_d+m-1),(j)]$$ where $a_i=\{1,m+1 \}$ for $i \in [d]$ have not been infected either. 

\vspace{1mm}

For the vertices on $B((1,2m),d,j)$ for $j \in \{2,3,\cdots,m-1\}$, the vertices on $B((1,m),d,j)$ have not been infected. Additionally, the vertices on the edges of $$[(a_1,a_1+m-1),(a_2,a_2+m-1),...,(a_d,a_d+m-1),(j)]$$ where $a_i=\{1,m+1 \}$ for $i \in [d]$ have not been infected either.

\vspace{2mm}

With the occurrence of the events $A_{1}^{d+1} $ and $B'$ then every vertex on $B((1,2m),d,i)$  for $i \in \{m+2,m+3,\cdots,2m-1\}$ has been infected after $m^{d+1}+C(d)m$ steps. except for those belonging to the sides of $[2m]^{d+1}$ and the vertices in $[(a_1),(a_2),\cdots,(a_{i-1}),(m),(a_{i+1}),\cdots, (a_d),(1,2m)]$ and $[(a_1),(a_2),\cdots,(a_{i-1}),(m+1),(a_{i+1}),\cdots, (a_d),(1,2m)]$ where $a_j \in \{1,2m\}$ for $j \in [d]$. 

\vspace{1mm}
Now note that every vertex in the subcube $B((1,2m),d,m+1)$ has at least one infected neighbor that is not on $B((1,2m),d,m+1)$ except for those that are in the sides of $B((1,2m(,d,m+1)$. So does every vertex in the subcube $B((1,2m),d,2m)$. Therefore we can use the induction to have
$$\mathbb{P}((E_{m+1}^d)^c \cap A_1^{d})\leq C(d)m^{d-3}(1-p)^{4m},$$ 
and 
$$\mathbb{P}((E_{2m}^d)^c \cap A_1^{d}) \leq C(d)m^{d-3}(1-p)^{4m}.$$ 

Since every vertex in the subcube $B((1,2m),d,m)$ has at least one infected neighbor that is not on $B((1,2m),d,m)$ except for those that are in the sides of $B((1,2m(,d,m+1)$, then the induction hypothesis can be applied and we have 
$$\mathbb{P}((E_{m}^{d})^c \cap A_1^{d}) \leq C(d)m^{d-3}(1-p)^{4m}.$$

\vspace{2mm}
 Inductively, every vertex on $[(1,2m),(1,2m),\cdots,(1,2m),(i)]$ for $i \in \{1,2,\cdots,m-1\}$ has been infected infected except for those on the edges of $B((1,2m),d,i)$ for $i \in \{1,2,\cdots,m-1\}$ after at most $2m^{d+1}+C(d)m^2$ steps. 

Now we need to deal with the uninfected vertices that are not in the sides of $[2m]^{d+1}$ after at most $2m^{d+1}+C(d)m^2$ steps. Let us describe them. For the vertices on $B((1,2m),d,i)$ for $i \in \{m+2,m+3,\cdots,2m-1\}$, the remaining uninfected vertices are on the sides of $[2m]^{d+1}$ and of the form $[(a_1),(a_2),\cdots,(a_{i-1}),(m),(a_{i+1}),\cdots, (a_d),(i)]$ and $[(a_1),(a_2),\cdots,(a_{i-1}),(m+1),(a_{i+1}),\cdots, (a_d),(i)]$ where $a_j \in \{1,2m\}$ for $j \in [d]$. 

\vspace{1mm}
For the vertices on $B((1,2m),d,j)$ where $j \in \{m+1,2m\} \cup [m]$, those remaining uninfected vertices are on the sides of $B((1,2m),d,j)$. More precisely, for the vertices on $B((1,2m),d,j)$, where $j \in \{2,3,\cdots,m-1\}$, the remaining uninfected vertices are $((1,m),(a_2),\cdots,(a_d),i)$ where $a_i=1$ for $i \in [d] \backslash \{1\}$ and of the form $[(a_1),(a_2),\cdots,(a_{i-1}),(m),(a_{i+1}),\cdots, (a_d),(i)]$ and $[(a_1),(a_2),\cdots,(a_{i-1}),(m+1),(a_{i+1}),\cdots, (a_d),(i)]$ where $a_j \in \{1,2m\}$ for $j \in [d] \backslash \{i\}$ and $[(a_1),(a_2),\cdots, (a_d),(i)]$ where $a_j \in \{1,2m\}$ for $j \in [d]$.

Let us consider the uninfected vertices on $\text{Perm}_{d+1}[(1,2m),\underbrace{(1),\cdots,(1)}_{d-1},(1,2m)]$. 
W.l.o.g, it is sufficient to consider the uninfected vertices on $[(1,2m),\underbrace{(1),\cdots,(1)}_{d-1},(1,2m)]$. 

Let $\mathscr{A}$ be event that both line segments 
$[(m+1,2m),(1),\cdots,(1),(m)]$
and 
$[(m),(1),(1),(m+1,2m)]$ are non-empty. 
This event is the same as the event $D(y=1)$ in the proof of Lemma \ref{lemma: E A1 B}. Therefore, the further analysis is the same as that in the proof of Lemma \ref{eq: P(E A1 d)}. 

The other uninfected vertices can be handled in the same way in Lemma \ref{lemma: E^d A^d}. 
\end{proof}

\begin{lemma}
The probability that $[2m]^d$ is bad along with the occurrence of the event that there are exactly two bad subcubes is correlated with double empty line segments of length $2m$. More precisely, we have 
  \begin{align*}
    \mathbb{P}((E^d)^c \cap A_2^d) \leq C(d)m^{d-1}(1-p)^{4m-8}. 
\end{align*}  
\end{lemma}
\begin{proof}
There are exactly two bad subcubes, i.e 
$$[(a_1,a_1+m-1),(a_2,a_2+m-1),...,(a_{d+1},a_{d+1}+m-1)]$$ where $a_i=\{1,m+1 \}$ for $i \in [d+1]$ are good except for exactly two of them.

\textbf{Case 1}. Assume that two bad subcubes are $[m]^{d+1}$ and 
\[
\begin{array}{c}
[(m+1,2m),\underbrace{(1,m),(1,m),\cdots, (1,m)}_{d}].
\end{array}
\]
After $m^{d+1}$ steps let us describe the uninfected vertices on $[2m]^{d+1}$. 

Every vertex on $B((1,2m),d,i)$ for $i \in \{m+2,m+3,\cdots,2m-1\}$ has been infected except for those of the form $[(a_1),(a_2),\cdots,(a_{d}),(i)]$, where $a_j \in \{1,m,m+1,2m\}$ for $j \in [d]$. \smallskip

For the vertices on $B((1,2m),d,j)$, where $j \in \{m+1,2m\}$, all of them have been infected except for those on the edges of $$[(a_1,a_1+m-1),(a_2,a_2+m-1),...,(a_d,a_d+m-1),(j)]$$ where $a_i=\{1,m+1 \}$ for $i \in [d]$. \medskip

For the vertices on $B((1,2m),d,j)$ where $j \in \{1,m\}$, the vertices on $B((1,m),d,j)$ and $[(m+1,2m),\underbrace{(1,m),\cdots, (1,m)}_{d-1},(j)]$ have not been infected. Additionally, the vertices on the edges of $$[(a_1,a_1+m-1),(a_2,a_2+m-1),...,(a_d,a_d+m-1),(j)]$$ where $a_i=\{1,m+1 \}$ for $i \in [d]$. \medskip

For the vertices on $B((1,2m),d,j)$ for $j \in \{2,3,\cdots,m-1\}$, the vertices on $B((1,m),d,j)$ and $[(m+1,2m),\underbrace{(1,m),\cdots, (1,m)}_{d-1},(j)]$  have not been infected. Additionally, the vertices of the form $[(a_1),(a_2),\cdots,(a_{d}),(j)]$, where $a_j \in \{1,m,m+1,2m\}$ for $j \in [d]$. 

\vspace{2mm}

Now assume that $B'$ occurs. Then after at most $m^{d+1}+C(d)m$ steps, every vertex on \\$[\underbrace{(1,2m),(1,2m),\cdots,(1,2m)}_{d},(j)]$ for $j \in \{m+2,\cdots,2m-1\}$ has been infected  except possibly for those on the edges of $[2m]^{d+1}$ and are of the form $[(a_1,a_2,\cdots,a_{i-1},m,a_{i+1},\cdots,a_d,j)]$ and $[(a_1,a_2,\cdots,a_{i-1},m+1,a_{i+1},\cdots,a_d,j)]$ where $a_i \in \{1,2m\}$ for $i \in [d]$.  

Therefore, every vertex on $B((1,2m),d,2m)$ and $B((1,2m),d,m+1)$ except for those on the edges of $B((1,2m),d,2m)$ and $B((1,2m),d,m+1)$ has one infected vertex that do not belong to $B((1,2m),d,2m)$ and $B((1,2m),d,m+1)$. We can use the induction to have that 
$$\mathbb{P}((E_{m+1}^d)^c \cap A_2^{d+1})\leq C(d)(1-p)^{4m},$$ 
and 
$$\mathbb{P}((E_{2m}^d)^c \cap A_2^{d+1}) \leq C(d)(1-p)^{4m}.$$ 

Similarly, every vertex on the subcube $B((1,2m),d,m)$ has an infected neighbor in  $B((1,2m),d,m+1)$ except for those that are in the edges of $B((1,2m),d,m)$. The induction hypothesis can be applied  and we have have 
$$\mathbb{P}(E_{m}^{d} \cap A_2^{d+1}) \leq C(d)m^{d-1}(1-p)^{4m-8}.$$

Similarly, then every vertex on the subcube $B((1,2m),d,m-1)$ has an infected neighbor in  $B((1,2m),d,m)$ except for those that are in the edges of $B((1,2m),d,m-1)$. The induction hypothesis can be applied to have 
$$\mathbb{P}(E_{m-1}^{d} \cap A_2^{d+1}) \leq C(d)m^{d-1}(1-p)^{4m-8}.$$

Inductively, we have for $i \in [m]$
$$\mathbb{P}(E_{i}^{d} \cap A_2^{d+1}) \leq C(d)m^{d-1}(1-p)^{4m-8}.$$


Let us describe the possibly uninfected vertices after $2m^{d+1}+C(d)m^2$ steps. 

\vspace{2mm}
The vertices on $B((1,2m),d,1)$ and $B((1,2m),d,2m)$ have been infected except for those on the edges of $[2m]^{d+1}$. 

\vspace{2mm}

The vertices on $B((1,2m),d,j)$ where $j \in \{2,3,\cdots,m,m+1\}$ have been infected except for those that are of the form 
$$[(a_1),(a_2),\cdots,(a_{i-1}),(1,2m),(a_{i+1}),\cdots, (a_{d}),(j)]$$
where $a_i=1$ for $i \in [d]$, and of the form
$$[(a_1),(a_2),\cdots,(a_{i-1}),(m),(a_{i+1}),\cdots, (a_{d}),(j)]$$
$$[(a_1),(a_2),\cdots,(a_{i-1}),(m+1),(a_{i+1}),\cdots, (a_{d}),(j)]$$
where $a_i \in \{1,2m\}$ for $i \in [d]$. 

\vspace{2mm}
The vertices on $B((1,2m),d,j)$ where $j \in \{m+2,m+3,\cdots,2m-1\}$ have been infected except for those of the form 
$$[(a_1),(a_2),\cdots,(a_{i-1}),(m),(a_{i+1}),\cdots, (a_{d}),(j)]$$
$$[(a_1),(a_2),\cdots,(a_{i-1}),(m+1),(a_{i+1}),\cdots, (a_{d}),(j)]$$
where $a_i \in \{1,2m\}$ for $i \in [d]$ and $j \in \{m+2,m+3,\cdots,2m-1\}$.

\vspace{2mm}
Let us consider the uninfected vertices on $\text{Perm}_{d+1}[(1,2m),\underbrace{(1),\cdots,(1)}_{d-1},(1,2m)]$.   Due to symmetry we only need to analyze the uninfected vertices on $[(1,2m),\underbrace{(1),\cdots,(1)}_{d-1},(1,2m)]$. The analysis is in the same fashion of that in the proof of Lemma \ref{lemma: E A2 B}, specifically in the case where two bad cubes are adjacent.

The other uninfected vertices can be handled in the same way as Lemma \ref{lemma: E^d A^d}. 

\textbf{Case 2.} Assume that two bad subcubes are $[m]^{d+1}$ and $[\underbrace{(m+1,2m),(m+1,2m),\cdots,(m+1,2m)}_{d+1}]$. 

Let us describe the uninfected vertices after $m^{d+1}$ steps. 

Every vertex on $B((1,2m),d,i)$ for $i \in \{m+2,m+3,\cdots,2m-1\}$ has been infected except for those on $[(m+1,2m),(m+1,2m),\cdots,(m+1,2m),(i)]$ and of the form $[(a_1),(a_2),\cdots,(a_{d}),(i)]$, where $a_j \in \{1,m,m+1,2m\}$ for $j \in [d]$. \smallskip

For the vertices on $[(1,2m),(1,2m),\cdots,(1,2m),(j)]$, where $j \in \{m+1,2m\}$, all of them have been infected except for those on $[(m+1,2m),(m+1,2m),\cdots,(m+1,2m),(j)]$ and 
the edges of $$[(a_1,a_1+m-1),(a_2,a_2+m-1),...,(a_d,a_d+m-1),(j)]$$ where $a_i=\{1,m+1 \}$ for $i \in [d]$. \medskip

    For the vertices on $B((1,2m),d,j^{'})$ where $j^{'} \in \{1,m\}$, the vertices on $B((1,m),d,j^{'})$ have not been infected. Additionally, the vertices on the edges of $$[(a_1,a_1+m-1),(a_2,a_2+m-1),...,(a_d,a_d+m-1),(j^{'})]$$ where $a_i=\{1,m+1 \}$ for $i \in [d]$ have not been infected. \medskip

For the vertices on $B((1,2m),d,j^{''})$ for $j^{''} \in \{2,3,\cdots,m-1\}$, the vertices on $B((1,m),d,j^{''})$ have not been infected. Additionally, the vertices of the form $[(a_1),(a_2),\cdots,(a_{d}),(j^{''})]$, where $a_j \in \{1,m,m+1,2m\}$ for $j \in [d]$, have not been infected. 

Assume that the event $D$ happens and then let us describe the uninfected vertices after $m^{d+1}+C(d)m$ steps. 

Every vertex on $B((1,2m),d,i)$ for $i \in \{m+2,m+3,\cdots,2m-1\}$ has been infected except for those of the form $[(a_1),(a_2),\cdots,(a_{d}),(i)]$, where $a_j \in \{1,2m\}$ for $j \in [d]$. Additionally, the vertices on $[(a_1),(a_2),\cdots,a_{j-1},(1,m),a_{j+1},\cdots,(a_{d}),(i)]$ for $i \in \{m+2,m+3,\cdots,2m-1\}$ where $a_j=2m$ for $j \in [d]$, have not been infected. \smallskip

For the vertices on $B((1,2m),d,j)$, where $j \in \{m+1,2m\}$, all of them have been infected except for those on the edges of $$[(a_1,a_1+m-1),(a_2,a_2+m-1),...,(a_d,a_d+m-1),(j)]$$ where $a_i=\{1,m+1 \}$ for $i \in [d]$. \medskip

For the vertices on $B((1,2m),d,j^{'})$where $j^{'} \in \{1,m\}$, the vertices on $[(1,m),(1,m),\cdots, (1,m),(j^{'})]$ have not been infected. Additionally, the vertices on the edges of $$[(a_1,a_1+m-1),(a_2,a_2+m-1),...,(a_d,a_d+m-1),(j^{'})]$$ where $a_i=\{1,m+1 \}$ for $i \in [d]$ have not been infected. \medskip

For the vertices on $B((1,2m),d,j^{''})$ for $j^{''} \in \{2,3,\cdots,m-1\}$, the vertices of the form $[(a_1),(a_2),\cdots,(a_{d}),(j^{''})]$, where $a_j \in \{1,2m\}$ for $j^{''} \in [d]$. Additionally, the vertices on $[(a_1),(a_2),\cdots,a_{j-1},(1,m),a_{j+1},\cdots,(a_{d}),(i)]$ for $i \in \{m+2,m+3,\cdots,2m-1\}$ where $a_j=1$, have not been infected. \smallskip

Then we can apply exactly the same method from the case where there is only one bad subcube to handle the rest of the uninfected vertices. \medskip 

\textbf{Case 3.} Assume that the two bad subcubes are $[m]^{d+1}$ and 
\[
\begin{array}{c}
[ \underbrace{(m+1,2m),(m+1,2m),\cdots, (m+1,2m) }_{l} \underbrace{,(1,m),(1,m),\cdots, (1,m),}_{d-l}(1,m)]
\end{array}
\]
where $ 2 \leq l \leq d$.

Then let us describe the uninfected vertices after $m^{d+1}$ steps. 

Every vertex on $B((1,2m),d,i)$ for $i \in \{m+2,m+3,\cdots,2m-1\}$ has been infected except for those of the form $[(a_1),(a_2),\cdots,(a_{d}),(i)]$, where $a_j \in \{1,m,m+1,2m\}$ for $j \in [d]$. \smallskip

For the vertices on $B((1,2m),d,j)$, where $j \in \{m+1,2m\}$, all of them have been infected except for those on the edges of $$[(a_1,a_1+m-1),(a_2,a_2+m-1),...,(a_d,a_d+m-1),(j)]$$ where $a_i=\{1,m+1 \}$ for $i \in [d]$. \medskip

For the vertices on $B((1,2m),d,j^{'})$ where $j^{'} \in \{1,m\}$, the vertices on $B((1,m),d,j^{'})$ and 
\[
\begin{array}{c}
[ \underbrace{(m+1,2m),(m+1,2m),\cdots, (m+1,2m) }_{l} \underbrace{,(1,m),(1,m),\cdots, (1,m),}_{d-l}(j^{'})]
\end{array}
\]
have not been infected. Additionally, the vertices on the edges of $$[(a_1,a_1+m-1),(a_2,a_2+m-1),...,(a_d,a_d+m-1),(j^{'})]$$ where $a_i=\{1,m+1 \}$ for $i \in [d]$ have not been infected. \medskip

For the vertices on $B((1,2m),d,j^{''})$ for $j^{''} \in \{2,3,\cdots,m-1\}$, the vertices on $B((1,m),d,j^{''})$ and 
\[
\begin{array}{c}
[ \underbrace{(m+1,2m),(m+1,2m),\cdots, (m+1,2m) }_{l} \underbrace{,(1,m),(1,m),\cdots, (1,m),}_{d-l}(j^{''})]
\end{array}
\]
have not been infected. Additionally, the vertices of the form $[(a_1),(a_2),\cdots,(a_{d}),(j^{''})]$, where $a_j \in \{1,m,m+1,2m\}$ for $j \in [d]$, have not been infected. 

Now assume that the event $D'$ occurs. Then let us describe the uninfected vertices after $m^{d+1}+C(d)m$ steps. 

Every vertex on $B((1,2m),d,i)$ for $i \in \{m+2,m+3,\cdots,2m-1\}$ has been infected except for those of the form $[(a_1),(a_2),\cdots,(a_{d}),(i)]$, where $a_j \in \{1,2m\}$ for $j \in [d]$. \smallskip

For the vertices on $B((1,2m),d,j)$, where $j \in \{m+1,2m\}$, all of them have been infected except for those on the edges of $$[(a_1,a_1+m-1),(a_2,a_2+m-1),...,(a_d,a_d+m-1),(j)]$$ where $a_i=\{1,m+1 \}$ for $i \in [d]$. \medskip

For the vertices on $B((1,2m),d,j^{'})$ where $j^{'} \in \{1,m\}$, the vertices on $[(1,m),(1,m),\cdots, (1,m),(j^{'})]$ and 
\[
\begin{array}{c}
[ \underbrace{(m+1,2m),(m+1,2m),\cdots, (m+1,2m) }_{l} \underbrace{,(1,m),(1,m),\cdots, (1,m),}_{d-l}(j^{'})]
\end{array}
\]
have not been infected. Additionally, the vertices on the edges of $$[(a_1,a_1+m-1),(a_2,a_2+m-1),...,(a_d,a_d+m-1),(j^{'})]$$ where $a_i=\{1,m+1 \}$ for $i \in [d]$ have not been infected. \medskip

For the vertices on $B((1,2m),d,j^{''})$ for $j^{''} \in \{2,3,\cdots,m-1\}$, the vertices on $B((1,m),d,j^{''})$ and 
\[
\begin{array}{c}
[ \underbrace{(m+1,2m),(m+1,2m),\cdots, (m+1,2m) }_{l} \underbrace{,(1,m),(1,m),\cdots, (1,m),}_{d-l}(j^{''})]
\end{array}
\]
have not been infected. Additionally, the vertices of the form $[(a_1),(a_2),\cdots,(a_{d}),(j^{''})]$, where $a_j \in \{1,m,m+1,2m\}$ for $j \in [d]$, have not been infected. 

Note that after $m^{d+1}+C(d)m$ steps, every vertex on $B((1,2m),d,j)$ where $j \in \{m+1,2m\}$ except for those on the edges of $B((1,2m),d,j)$ where $j \in \{m+1,2m\}$ has one infected neighbor which does not belong to $B((1,2m),d,j)$ where $j \in \{m+1,2m\}$. Thus we can use the induction hypothesis to deal with the uninfected vertices that do not belong to the edges of $B((1,2m),d,j)$ where $j \in \{m+1,2m\}$. 

Similarly for the uninfected vertices on $B((1,2m),d,j)$ where $j \in [m]$ we can use the induction hypothesis to deal with the uninfected vertices that do not belong to the edges of $B((1,2m),d,j)$ where $j \in [m]$. 

For the rest of the uninfected vertices, we can use the same approach from the case where there is only one bad subcube to handle. 
\end{proof}

Now we will use the inequality derived in Lemma \ref{lemmaV3} to derive an important property of the grid size $L$, which states that the probability percolation fails to happen in a grid $[L]^d$ correlates to the probability of the existence of an empty double line segments of length $L$. 

Before stating the following lemma, we need to define a quantity $K=K(p)$ which will be used multiple times from now on in this section. Define 
\begin{align*}
K(p):=
\begin{cases}
  \exp^{(d-1)}(\frac{2\lambda}{p}) & \text{if $p \leq p_0$} \\
    \exp^{(d-1)}(\frac{2\lambda}{p_0}) & \text{if $p \geq p_0$,}
\end{cases}
\end{align*}
where $\exp^{l}(\cdot)$ denotes iterating the exponential function $l$ times and $p_0$ will be defined later.

\begin{lemma} \label{lemmaV4}
    If $L \geq 16K^3 \left(\log \frac{1}{1-p} \right)^2 \left(\frac{1}{\log\frac{1}{\delta}} \right)^2$, then the probability $\eta_L$ that a grid $[L]^d$ is bad satisfies 
   $$\eta_L \leq B(d)L^{d-1}(1-p)^{2L-8},$$
   where $B(d)>0$ and $\delta \leq \frac{1}{C(d)}$. 
\end{lemma}

\begin{proof}

We will make use of Theorem~\ref{theoremV3}. Since $p_c([n]^d,r=d)=\frac{\lambda(1+o(1))}{\log^{(d-1)} n}$, by taking $p_0$ small enough, 
we have $$R([K]^d,r=d,p)=1-o(1),$$
where $R([K]^d,d,p)$ is the probability that every vertex on $[K]^d$ will be infected by the percolation process with the infection threshold $r=d$ if each vertex is initially infected with probability $p$.

Therefore, let $\delta >0$ and we have 
$$\eta_K \leq \delta.$$

From Lemma~\ref{lemmaV3}, we have 
$$\eta_{2m} \leq C(d)\eta_m^3+B(d)m^{d-1}(1-p)^{4m-8},$$ 
and thus 
$$\eta_{2m} \leq 2\max \{C(d)\eta_m^3,B(d)m^{d-1}(1-p)^{4m-8} \}$$

We will use $C$ and $B$ to denote $C(d)$ and $B(d)$ respectively in the proof. 

Note that if 
\begin{align} \label{eq: eta}
C\eta_m^3 \le Bm^{d-1}(1-p)^{4m-8}, 
\end{align}
then we have 
$$\eta_{2m} \le Bm^{d-1}(1-p)^{4m-8},$$ and thus the desired result.

From the recursive relation $\eta_{2m} \leq C\eta_m^3$ with the initial condition $\eta_K \leq \delta,$ we have 
$$\eta_{2^rK} \leq C^{\frac{3^r-1}{2}} \delta^{3^r}.$$

In order to satisfy (\ref{eq: eta}) it suffices to have  
$$C^{\frac{3^r-1}{2}} \delta^{3^r} \leq B(2^{r-1}K)^{d-1}(1-p)^{4\times 2^{r-1}K-8},$$ 
which is equivalent to
\begin{align} \label{equation1}
    \frac{3^r-1}{2}\log C- 3^r \log \frac{1}{\delta} \leq \log B +(d-1)  \log (2^{r-1}K)-2^{r+1}K \log \frac{1}{1-p}.
\end{align}

Assume $\delta \leq \frac{1}{C}$ by taking $p_0$ sufficeintly small. Then it is easy to see that as long as 
$$\frac{3^r}{2} \log \frac{1}{\delta} \geq 2^{r+1}K \log \frac{1}{1-p},$$ is satisfied,  (\ref{equation1}) is satisfied.

Therefore, we have 
$$\frac{3^r}{2^r} \geq 4K \left(\log \frac{1}{1-p} \right) \left(\frac{1}{\log \frac{1}{\delta}} \right),$$ which is equivalent to 
$$\left(\frac{3}{2} \right)^{r\log_{\frac{3}{2}}2}K \geq K\left(4K \left(\log \frac{1}{1-p} \right) \left(\frac{1}{\log \frac{1}{\delta}} \right) \right)^{\log_{\frac{3}{2}}2}.$$

Thus if $L \geq K \left(4K \left(\log \frac{1}{1-p} \right) \left(\frac{1}{\log \frac{1}{\delta}} \right) \right)^{\log_{\frac{3}{2}}2}$, then 
$$\eta_L \leq BL^{d-1}(1-p)^{2L-8}. \qedhere $$
\end{proof}

Now fix $L=16K^3\left(\log \frac{1}{1-p} \right)^2 \left(\frac{1}{\log \frac{1}{\delta}} \right)^2$. 
 
Let us provide some explanations before stating and proving the following lemma. We would like to estimate the probability that a particular vertex, say the origin, is uninfected at time $t$. We show that this particular probability is roughly the same as the probability of having an empty line segment of length $t$ starting on the origin. As mentioned in the introduction, the existence of an initially uninfected $[2t+1] \times [2]^{d-1}$ rectangle "near" the origin implies that the origin will not be infected at time $t$. Thus the following lemma is essentially saying that the probability from other configurations that also prevent the origin from being infected at time $t$ is negligibly small.

\begin{lemma} \label{lemmaV5}
    Let $t$ be an integer and $t'=L^d$. Every vertex of $[n]^d$ is initially infected with probability $p$ independent of any other vertex. Then 
    $$\mathbb{P}( \text{the origin is uninfected at time t} ) \leq C(d)\frac{(1-p)^{t-t'}}{p}.$$

\end{lemma}
\begin{proof}
Let us introduce some notation before proceeding. Let $[(a)^k,(b)^l]:=[\underbrace{(a),\cdots,(a)}_{k},\underbrace{(b),\cdots,(b)}_{l}]$.

Suppose that the origin is uninfected at time $t$. Then at least one of the vertices among $[(1),(0)^{d-1}]$, $[(0),(1),(0)^{d-2}]$ ,$\cdots$, and $[(0)^{d-1},(1)]$ has to be uninfected at time $t-1$ since the origin has $2d$ neighbors and the infection threshold $d$. Without loss of generality if $[(1),(0)^{d-1}]$ is uninfected at time $t-1$ then one of the vertices among $[(2),(0)^{d-1}]$ , $[(1),(1),(0)^{d-2}]$ ,$\cdots$, and $[(1),(0)^{d-2},(1)]$ \} has to be uninfected at time $t-2$ by the same reason. It is easy to see that there has to exist a path starting at the origin and moving along only in the directions of the standard basis vectors $e_1,e_2,\cdots,e_d$ of initially uninfected vertices of length $t$.

Therefore, there exists a path $y_1,y_2,\cdots,y_{t-t'}$, starting at the origin, of initially uninfected vertices of length $t-t'$ at time $t'$. Let $M >0$ be the maximal such that at least two coordinates of $y_{t-t'-M}$ are non-zero. By definition if $M=0$, then the path $y_1,y_2,\cdots,y_{t-t'}$ starts at the origin and keeps going straight along the standard basis vector $e_1$, $e_2$,..., or $e_d$. We will show that the most likely way to guarantee the origin to be uninfected at time $t$ is to have the path $y_1,y_2,\cdots,y_{t-t'}$ starting at the origin and parallel to the standard basis vector $e_1$, $e_2$, $\cdots$ , or $e_d$ which corresponds to $M=0$. This corresponds to the heuristic that the minimal configuration of uninfected vertices for the origin being uninfected at time $t$ is the most likely way to guarantee the origin to be uninfected at time $t$.

The path $y_{t-t'-M},\cdots, y_{t-t'}$ intersects an $L$-path which consists of disjoint cubes $D_1,D_2,...,D_l$ of size $[L]^d$ with $D_1=[(y_{t-t'-M}^1,y_{t-t'-M}^1+L-1),(y_{t-t'-M}^2,y_{t-t'-M}^2+L-1),...,(y_{t-t'-M,}^d,y_{t-t'-M}^d+L-1)]$ where 
$y_{t-t'-M}=(y_{t-t'-M}^1,y_{t-t'-M}^2,\cdots, y_{t-t'-M}^d)$. It is easy to see that $\frac{M}{L}-2 \leq l \leq \frac{M}{L}.$

The cubes $D_1,D_2,...,D_l$ are either semi-good or bad since at time $t'=L^d$, $y_1,...y_{t-t'}$ are uninfected.  

Let $F_2(i,j)$ denote the event that the cubes $ D_i, \cdots, D_j$ are semi-good, and that there exists a path $y_k, \cdots, y_h$ of uninfected vertices at time $t'$, which is entirely contained in the sides of these cubes. The initial vertex $y_k$ lies in one of the following sets:
\[
[(a_1, a_1+L-1), (a_2), \ldots, (a_d)],\quad
[(a_1), (a_2, a_2+L-1), \ldots, (a_d)],\quad \cdots \; \text{or} \quad
[(a_1), (a_2), \ldots, (a_{d-1}), (a_d, a_d+L-1)],
\]
where
\[
D_i = [(a_1, a_1+L-1), (a_2, a_2+L-1), \ldots, (a_d, a_d+L-1)].
\]

Similarly, the final vertex \( y_h \) lies in one of the following sets:
\[
[(b_1 - (L-1), b_1), (b_2), \ldots, (b_d)],\quad
[(b_1), (b_2 - (L-1), b_2), \ldots, (b_d)],\quad \text{or} \quad
[(b_1), (b_2), \ldots, (b_{d-1}), (b_d - (L-1), b_d)],
\]
where
\[
D_j = [(b_1 - (L-1), b_1), (b_2 - (L-1), b_2), \ldots, (b_d - (L-1), b_d)].
\]


If a path of uninfected vertices goes through one side $[(a_1,a_1+l-1),(a_2),...,(a_d)]$ of a cube of size $[L]^d$ at time $t'$ and the interior of this cube has been infected at time $t'$, then the sides $[(a_1,a_1+l-1),(a_2),(a_3),...,(a_d)]$, $[(a_1,a_1+l-1),(a_2+1),(a_3),...,(a_d)]$,...,$[(a_1,a_1+l-1),(a_2),(a_3),...,(a_d+1)]$ will have to be uninfected. Since the interior of this cube has been infected by time $t'$ every vertex on the sides has $d-1$ infected neighbors by the time $t'$. 

We need to give a definition before proceeding. Given a cube $D=[(a_1,b_1),(a_2,b_2),\cdots,(a_d,b_d)]$ where $b_i-a_i=m-1$ for $i \in [d]$, let a buffer of $D$ for the side $[(a_1),(a_2,b_2),(a_3),\cdots,(a_d)]$ be the $ [2] \times [m-2] \times \underbrace{[1] \cdots \times [1]}_{d-2}$ rectangle $[(a_1-1,a_1),(a_2+1,b_2-1),(a_3),\cdots,(a_d)]$. Define the set of buffers of $D$ for the side $[(a_1),(a_2,b_2),(a_3),\cdots,(a_d)]$ to be the set 
\begin{align*}
\{ &[(a_1-1,a_1),(a_2+1,b_2-1),(a_3),\cdots,(a_d)],\\
   &[(a_1),(a_2+1,b_2-1),(a_3-1,a_3),\cdots,(a_d)],\\
   &\cdots \\
   &[(a_1),(a_2+1,b_2-1),(a_3),\cdots,(a_d-1,a_d)]\}
 \end{align*}  
The set of buffer for the other sides of $D$ is defined similarly. Let $\mathbf{B}$ be the set of buffers of $D_i,\cdots, D_{i+r-1}$.

Now suppose that there is a consecutive $r$ semi-good cubes $D_i,...,D_{i+r-1}$ so that the event $F_2(i,i+r-1)$ happens. Since the interior of the cubes $D_i,...,D_{i+r-1}$ have been infected at time $t'$, the existence of a path of uninfected vertices along the standard basis vectors $e_1,e_2,\cdots,e_d$ going along the sides of $D_i,\cdots,D_{i+r-1}$ implies that 
at least $r-1$ of the buffers in $\mathbf{B}$ are uninfected. Since each buffer is a set of $d(L-2)$ vertices, we have 
$$\mathbb{P}(F_2(i,i+r-1)) \leq d^{r-1}(1-p)^{d(r-1)(L-2)},$$
If $r=1$, we have 
$$\mathbb{P}(F_2(i,i)) \leq \frac{2^dd}{2} (1-p)^L.$$
Indeed since $D_i$ is semi-good at least one of its sides is initially uninfected. 

Let $F_1(i,j)$ denote the event that the cubes $D_i,...D_j$ are bad. From Lemma ~\ref{lemmaV4}, we have 
$$\mathbb{P}(F_1(i,i+r-1)) \leq (BL^d(1-p)^{-8})^{r}(1-p)^{2Lr}$$

There exists a finite sequence $0=b_1 < s_1 < b_2 < s_2 < b_3 <...,$ where the last term is $l$, such that the event 
$$F=F_1(b_1,s_1-1) \cap F_2(s_1,b_2-1)\cap F_1(b_2,s_2-1) \cap  F_2(s_2,b_3-1) \cap ...$$
happens. Assume that the last term of the sequence is $l=b_{u+1}-1.$ Let $v$ be the number of $i$ such that $b_{i+1}-1=s_i$, i.e, $v$ is the number of times that there are three consecutive cubes in the sequence $D_1,...,D_l$ that are of the form bad, semi-good, bad. We have
\begin{align*}
    \mathbb{P}(F) &= \prod_{i=1}^u \mathbb{P}\left( F_1(b_i,s_i-1) \cap F_2(s_i,b_{i+1}-1) \right) \\
    & \leq (2^{d-1}d)^v (1-p)^{Lv} \prod_{i=1}^u g(p)^{s_i-b_i}(1-p)^{2L(s_i-b_i)} d^{b_{i+1}-s_i-1}(1-p)^{d(L-2)(b_{i+1}-s_i-1)} \\
    & \leq (2^{d-1}d)^v (1-p)^{Lv} \prod_{i=1}^u g(p)^{s_i-b_i}(1-p)^{2L(s_i-b_i)} d^{b_{i+1}-s_i-1}(1-p)^{2(L-2)(b_{i+1}-s_i-1)} \\
    & \leq (2^{d-1}d^2g(p))^l (1-p)^{(L-2)(2l-2u+v)},
\end{align*}
where $g(p)=BL^d(1-p)^{-8}.$

Moreover, we have $2v+3(u-v) \leq l$ by partitioning sequences of consecutive semi-good cubes of size $[L]^d$ into those of length 1 and those of length greater than 1. Thus we have $2u-v \leq \frac{2l}{3}$. Therefore,
$$\mathbb{P}(F) \leq (2^{d-1}d^2g(p))^l (1-p)^{\frac{4(L-2)l}{3}}.$$ 

For a given $l$, there are $d^l$ choices of the path along along the standard basis vector $e_1$, $e_2$,..., or $e_d$ of cubes of size $[L]^d$ and $2^l$ ways of choosing whether each cube is good or bad. Therefore, if we let H be an event that there exists a path along $e_1,e_2$,...,or $e_d$ direction of length $M$ starting from a given vertex, then we have 
$$\mathbb{P}(H) \leq (2^dd^3g(p))^{\frac{M}{L}}(1-p)^{\frac{4M(1-\frac{2}{L})}{3}} \leq (2^d d^3g(p))^{\frac{M}{L}}(1-p)^{1.3M}$$

Therefore, with $G$ being an event that the origin is uninfected at time $t$, we have 
\begin{align}
    \mathbb{P}(G) & \leq d \sum_{M=0}^{t-t'}(1-p)^{t-t'-M}(2^dd^3g(p))^{\frac{M}{L}}(1-p)^{1.3M} \\ 
    & \leq \frac{C(1-p)^{t-t'}}{1-(1-p)^{0.2}} \\ \label{eq: (1-p)^(t-t')}
    & \leq \frac{C(1-p)^{t-t'}}{p} 
\end{align}
The inequality \ref{eq: (1-p)^(t-t')} is true since $(2^dd^3g(p))^{\frac{1}{L}}(1-p)^{0.3} \leq (1-p)^{0.2}$ and $p_0$ can be sufficiently small. \qedhere

\end{proof}

Now we are ready to prove the upper bound for the percolation time. 
\begin{proof}
Let $t'=L^d$ and $H'$ be an event that there exists a vertex in $[n]^d$ which is uninfected at time $t$.
We have
\begin{align*}
    \mathbb{P}(T \geq t) &= \mathbb{P}(H')\\
    &\leq  n^d \frac{(1-p)^{t-t'}}{p} \\
    & \leq \exp \left(d \log n+\log \frac{1}{p}-(t-t')\log \frac{1}{1-p} \right)\\
    & =o(1), 
\end{align*}
if  $\frac{d \log n}{\log \frac{1}{1-p}} \leq t-t'$.
 
Therefore, with high probability 
$$T \leq C \frac{\log n}{\log \frac{1}{1-p}},$$
since $t'=o\left(\frac{\log n}{\log \frac{1}{1-p}} \right)$. 

Thus the upper bound is proved. \qedhere
\end{proof}

\section{Open problems}
In this paper we have extended one of the two main theorems in \cite{Paul} to the higher dimensional case when the initial infection probability $p(n)$ is large and the infection threshold $r=d$. There are a few questions that can be further investigated. 

Question 1: What is the distribution of the percolation time $T$ on $[n]^d$ if $ \frac{\lambda}{\log^{d-1}n} \leq p(n) \leq \frac{C}{\log^{d}(n)}$ and $r=d$?

In \cite{Paul}, this question was addressed for the case $d = 2$ and $r = 2$. The main difficulty in adapting the ideas from \cite{Paul} to higher dimensions lies in the fact that the percolation process behaves quite differently when $d = r \geq 3$ compared to the two-dimensional case. For $d = r = 2$, the results in \cite{Paul} rely heavily on the so-called rectangular process. In higher dimensions, however, it is unclear how to generalize this notion, and indeed, all existing results on the critical probability rely on induction on the dimension $d$. Nevertheless, to analyze the percolation time in higher dimensions, this inductive approach appears to be insufficient. 

Question 2: What is the distribution of the percolation time $T$ on $[n]^d$ if $r < d$ and $p(n) \geq \frac{\lambda}{\log^{d}(n)}$? 

This question remains largely open, and very little is known except in the case where the initially infected set is very dense, i.e., when $p(n)$ is close to $1$, as studied in \cite{Bollobs_bootstrap} and \cite{Bollobs2012TheTO}. The first nontrivial case arises when $d = 3$ and $r = 2$.

\section*{Acknowledgment}
The author would like to thank Prof.~Alexander Barg for suggesting this problem and providing valuable guidance. Additionally, the author extends thanks to Yihan Zhang for insightful discussions.

\bibliographystyle{plain}
\bibliography{main.bib}

\end{document}